\documentclass[a4paper,11pt]{amsart}
\usepackage{amssymb}
\usepackage[OT1]{fontenc}
\usepackage[latin1]{inputenc}
\usepackage{hyperref}
\usepackage[curve]{xypic}
\usepackage{enumerate}
\usepackage[active]{srcltx}
\usepackage{ae,aecompl}
\usepackage{xcolor}
\usepackage{mathtools}

\usepackage{gastex}
\usepackage{graphicx}
\usepackage{auto-pst-pdf}

\usepackage{pgf}
\usepackage{tikz}
\usetikzlibrary{positioning,automata,shapes}

\def\pv#1{\ensuremath{{\mathsf{#1}}}}

\def\F#1{\ensuremath\widehat{#1^\ast}}
\def\FG#1{\ensuremath\widehat{F_G}(#1)}

\def\Syn{\ensuremath{M}}

\newtheorem{Thm}{Theorem}[section]
\newtheorem{Prop}[Thm]{Proposition}
\newtheorem{Lemma}[Thm]{Lemma}
\newtheorem{Cor}[Thm]{Corollary}

\newtheorem{Example}[Thm]{Example}

{\theoremstyle{remark}
\newtheorem{Rmk}[Thm]{Remark}}

\numberwithin{equation}{section}

\begin{document}

\title[On the group of a rational maximal bifix code]{On the group of a
  rational maximal bifix code}

\thanks{%
  J.~Almeida and A.~Costa acknowledge partial funding by CMUP
  (UID/MAT/ 00144/2013) and CMUC (UID/MAT/00324/2013), respectively,
  which are funded by FCT (Portugal) with national (MCTES) and
  European structural funds (FEDER) under the partnership agreement
  PT2020. %
  The work of J.~Almeida was carried out in part at Masaryk
  University, whose hospitality is gratefully acknowledged, with the
  support of the FCT sabbatical scholarship SFRH/BSAB/142872/2018.}

\author[J.~Almeida]{Jorge Almeida}
\address{CMUP, Departamento de Matem\'atica,
     Faculdade de Ci\^encias, Universidade do Porto, 
 Rua do Campo Alegre 687, 4169-007 Porto, Portugal.}
\email{jalmeida@fc.up.pt}

\author[A.~Costa]{Alfredo Costa}
\address{CMUC, Department of Mathematics, University of Coimbra,
  Apartado 3008, EC Santa Cruz, 3001-501 Coimbra, Portugal.}
\email{amgc@mat.uc.pt}

\author[R.~Kyriakoglou]{Revekka Kyriakoglou}
\address{LIGM, Universit\'e Paris-Est}
\email{krevekka@hotmail.gr}

\author[D.~Perrin]{Dominique Perrin}
\address{LIGM, Universit\'e Paris-Est}
\email{dominique.perrin@esiee.fr}



\begin{abstract}
  We give necessary and sufficient conditions for
  the group of a rational maximal bifix code $Z$ to be isomorphic
  with the $F$-group of \mbox{$Z\cap F$}, when $F$ is recurrent and
  $Z\cap F$ is rational. The case where
  $F$ is uniformly recurrent, which is known to imply the finiteness of
  \mbox{$Z\cap F$}, receives special attention.
  The proofs are done by exploring the connections with the structure of the free profinite monoid over the alphabet of $F$.
\end{abstract}

\maketitle

\section{Introduction}

Maximal bifix codes (and, more generally, maximal prefix codes and maximal suffix codes) play a central role in the theory of codes~\cite{Berstel&Perrin&Reutenauer:2010}. In the past few years,
special attention has been given to (bifix, prefix, suffix) codes
which may not be maximal but are maximal within some language,
usually
a recurrent or uniformly recurrent language. This line of research has produced new and strong connections between bifix codes, subgroups of free groups and symbolic dynamical systems
(see the survey~\cite{Perrin:2018}
and the series of papers~\cite{Berstel&Felice&Perrin&Reutenauer&Rindone:2012,Berthe&Felice&Dolce&Leroy&Perrin&Reutenauer&Rindone:2015,Berthe&Felice&Dolce&Leroy&Perrin&Reutenauer&Rindone:2015b,Berthe&Felice&Dolce&Leroy&Perrin&Reutenauer&Rindone:2015c,Berthe&Felice&Dolce&Leroy&Perrin&Reutenauer&Rindone:2015d}).

If $Z$ is a thin maximal bifix code and $F$ is a recurrent set, then
$X=Z\cap F$ is an $F$-maximal bifix code, that is,
a maximal bifix code within $F$~\cite{Berstel&Felice&Perrin&Reutenauer&Rindone:2012}. Moreover, $X$ is finite if
$F$ is uniformly recurrent.
One can informally speak of a  process of relativization, or localization,
going from maximal bifix codes to $F$-maximal bifix codes, when $F$ is recurrent.
Inevitably, the study of $F$-maximal bifix codes
leads to a process of relativization of several previously known
notions and results about maximal bifix codes.
This paper is focused in one of these notions. The group of a rational code $Z$, denoted by $G(Z)$, is the Sch\"utzenberger group of the minimum ideal of
the syntactic monoid $\Syn(Z^\ast)$ of $Z^\ast$. This group is an important parameter in the study of rational codes~\cite{Berstel&Perrin&Reutenauer:2010}.
The relativization of this parameter consists in taking
the intersection $X=Z\cap F$ and the Sch\"utzenberger group of the minimum $\mathcal J$-class
that intersects the image of $F$ in the syntactic monoid of
$X^\ast$. This group, denoted by $G_F(X)$, is the \emph{$F$-group of $X$}.
How are $G(Z)$ and $G_F(X)$ related? They are not always isomorphic,
even if $Z$ is a group code and $F$ is uniformly recurrent.
In~\cite{Berstel&Felice&Perrin&Reutenauer&Rindone:2012} it is shown that
if $Z$ is a group code  and $F$ is a Sturmian set, then
$G(Z)$ and $G_F(X)$ are isomorphic. This result was extended
to arbitrary uniformly recurrent tree sets in the manuscript~\cite{Kyriakoglou&Perrin:2017}. That was done with a novel approach consisting in exploring and applying links between $G(Z)$, $G_F(X)$
and the Sch\"utzenberger (profinite) group $G(F)$
of the minimum $\mathcal J$-class $J(F)$
of the topological closure of $F$ within the free profinite monoid generated by the alphabet of~$F$, and by taking advantage, with the help of these links,
of results on $G(F)$ established in~\cite{Almeida&ACosta:2013}
and~\cite{Almeida&ACosta:2016b}.

Free profinite monoids have proved in the last few decades to be of major importance in the study
of formal (rational) languages~\cite{Almeida:1994a,Rhodes&Steinberg:2009qt}.
Their elements, called pseudowords, can be seen as generalizations of words,
but the algebraic structure of free profinite monoids is much richer than
that of free monoids. The first author established a connection
with symbolic dynamics~\cite{Almeida:2004a} which led to research on
the $\mathcal J$-classes of the form $J(F)$, when $F$ is recurrent, and of their maximal subgroups,
thereby elucidating structural aspects of free profinite monoids~\cite{Almeida:2004a,ACosta&Steinberg:2011,Almeida&ACosta:2013,Almeida&ACosta:2016b}.
The approach followed in~\cite{Kyriakoglou&Perrin:2017}
indicates that it is
worthwhile to
extend to the theory of codes this connection
between free profinite monoids and recurrent sets. 
In this paper we corroborate this, by further exploring the relationship between $G(Z)$, $G_F(X)$ and~$G(F)$. We do it
inspired by~\cite{Kyriakoglou&Perrin:2017}, but without
depending on results first appearing in that~manuscript.

Our main result (Theorem~\ref{t:group-code-isomor-maximal-subgroups}) gives necessary and sufficient conditions for the isomorphism $G(Z)\simeq G_F(X)$
when $Z$ is a rational maximal bifix code, $F$ is recurrent, and $Z\cap F$ is rational.
When $Z$ is a group code and $F$ is uniformly recurrent, we recover the
corresponding result from the manuscript~\cite{Kyriakoglou&Perrin:2017}, which in this paper is slightly improved (cf.~Theorem~\ref{t:uniformly-recurr-version-group-code-isomor-maximal-subgroups}).
Moreover, we deduce
the isomorphism $G(Z)\simeq G_F(X)$
when $Z$ is a group code and $F$ is a uniformly recurrent connected set (cf.~Corollary~\ref{t:equal-degree-connected-set-group-code}).
These results are framed under the new notion of \emph{$F$-charged code}.

The paper is organized as follows. Following the introduction, we
have a section of preliminaries.
It is divided in several subsections, in order to
encompass codes, recurrent sets, free profinite monoids and syntactic monoids, and connections between all of these. The section contains some preparatory results
needed for our main contributions.
In the third section, we further develop the extension to pseudowords,
initiated in~\cite{Kyriakoglou&Perrin:2017}, of the key notion of parse of a word, considering now arbitrary rational codes. In particular, we study the continuity
of the function that counts the number of parses, with respect to the discrete topology of the set of natural numbers. This facilitates
the development of the fourth section, which presents the main results.
In this section, preference
 was given to the more algebraic
 characterization of
 the syntactic monoid as the monoid of classes of words with the same contexts,
 and its variation in terms of pseudowords (cf.~Lemma~\ref{l:syntactic-inequality-profinite-version}).
 In the manuscript~\cite{Kyriakoglou&Perrin:2017} the characterization of the syntactic monoid as the transition monoid of the minimal automaton seems to be more notorious. That perspective
 is also explored in the fifth section of this paper,
 where we look at $G(Z)$ and $G_F(X)$
 as permutation groups and establish conditions
 under which they are equivalent.

\section{Preliminaries}
\label{sec:preliminaries}

\subsection{Codes contained in factorial sets of words}

Along this paper, all alphabets are finite, $A^*$ denotes
the free monoid generated by the alphabet~$A$, the empty word is denoted by $1$, and
$A^+=A^*\setminus\{1\}$ is the free semigroup generated by~$A$. 

Recall that a \emph{code} of $A^*$ is a subset $X$ of $A^+$ such that
the submonoid of $A^*$ generated by $X$ is freely generated by $X$.

Let $X$ be a nonempty subset of $A^+$.
If $X\cap XA^+=\emptyset$,
then $X$ is a code, and it is said to be a \emph{prefix code}.
Dually, if $A^+X\cap X=\emptyset$, then 
$X$ is a \emph{suffix code}.
A \emph{bifix code} is a code that is simultaneously a prefix and a suffix code.

See~\cite{Berstel&Perrin&Reutenauer:2010} for a systematic study of codes. 
For the purposes of this paper, the paper~\cite{Berstel&Felice&Perrin&Reutenauer&Rindone:2012} is also a useful reference.

Let $F$ be a nonempty subset of $A^*$. A prefix code $X$ is an \emph{$F$-maximal prefix code}
if $X\subseteq F$ and $X$ is not properly contained in any other prefix code contained in $F$. Replacing the word ``prefix'' by ``suffix'' or by ``bifix'',
we obtain the notions of \emph{$F$-maximal suffix code} and
\emph{$F$-maximal bifix code}, respectively.
A subset $X$ contained in $F$
is \emph{right $F$-complete}
if every element of $F$ is a prefix of an element of
$X^\ast$. The next proposition
is an immediate application of the combination of Propositions 3.3.1
and 3.3.2 from~\cite{Berstel&Felice&Perrin&Reutenauer&Rindone:2012},
which in turn are extensions of results from~\cite{Berstel&Perrin&Reutenauer:2010}
established for the special case $F=A^*$.
Recall that a subset $F$ of $A^*$  is \emph{factorial} if it contains
every word of $A^*$ that is a factor of at least one element of $F$.

\begin{Prop}\label{p:right-F-complete}
  Suppose that $F$ is a factorial subset of $A^*$, and let $X$
  be a prefix code contained in $F$. Then $X$ is an $F$-maximal prefix code
  if and only if $X$ is right $F$-complete.
\end{Prop}

There is a dual definition of \emph{left $F$-complete} subset of $A^*$, and
a result for suffix codes that is dual to Proposition~\ref{p:right-F-complete}.

\subsection{Recurrent and uniformly recurrent sets}

We say that a factorial subset $F$ of $A^*$ is \emph{recurrent} if
$F\neq \{1\}$ (in the definition given in~\cite{Berstel&Felice&Perrin&Reutenauer&Rindone:2012}
one may have $F=\{1\}$) and for every $u,v\in F$ there is $w\in F$ such that $uwv\in F$.
A recurrent set is said to be \emph{uniformly recurrent} if, for every $u\in F$,
there is a positive integer $n$ such that $u$ is a factor of every
element of $F$ with length at least $n$.
A special case of a uniformly recurrent set is that of
a \emph{periodic set}, that is a set which is the set
of factors of a language of the form  $u^\ast$, with $u$ a nonempty word.
These notions appear in the field of symbolic dynamics,
for which we give \cite{Lind&Marcus:1996,Lothaire:2001,Fogg:2002}
as references.
Indeed, a subset $F$ of $A^*$ is recurrent (respectively, uniformly recurrent) if and only if it is the language of finite words appearing in the
elements of an irreducible symbolic dynamical system
(respectively, of a minimal symbolic dynamical system)
of $A^{\mathbb Z}$, and $F$ is periodic
when the corresponding symbolic dynamical system is periodic.

We briefly describe an important mechanism for producing
uniformly recurrent sets. A \emph{substitution}
$\varphi$ over a finite alphabet $A$
is an endomorphism of $A^*$.
If there is $n$ such that, for all $a\in A$, every letter of $A$ is
a factor of $\varphi^n(a)$, then $\varphi$ is said to be \emph{primitive}.
If $\varphi$ is primitive and not the identity on the free monoid over a one-letter alphabet,
then the set $F(\varphi)$ of factors of elements of $\{\varphi^k(a)\mid a\in A,k\geq 1\}$ is a uniformly recurrent subset of $A^*$.

\begin{Example}\label{eg:fibonacci-set}
  Let $A=\{a,b\}$ and let $\varphi$ be the substitution
  over $A$ given by $\varphi(a)=ab$ and $\varphi(b)=a$.
  This substitution is primitive. It is called
  the \emph{Fibonacci substitution}.
  The uniformly recurrent set $F(\varphi)$ is
  the \emph{Fibonacci~set}.
\end{Example}

\begin{Example}\label{eg:tribonacci-set}
  A related example is the \emph{Tribonacci substitution},
  the substitution $\psi$ over $A=\{a,b,c\}$,
  defined by $\psi(a)=ab$, $\psi(b)=ac$ and $\psi(c)=a$.
  The corresponding uniformly recurrent set
  $F(\psi)$ is the \emph{Tribonacci~set}.
\end{Example}

\subsection{The extension graph}

Consider a recurrent subset $F$ of $A^*$.
Given $w\in F$, let
\begin{align*}
  L(w)&=\{a\in A\mid aw\in F\},\\
  R(w)&=\{a\in A\mid wa\in F\},\\
  E(w)&=\{(a,b)\in A\times A\mid awb\in F\}.
\end{align*}
The \emph{extension graph $G(w)$} is the bipartite undirected graph whose
vertex set is the union of disjoint copies of $L(w)$ and $R(w)$,
and whose edges are the pairs $(a,b)\in E(w)$, with incidence in
$a\in L(w)$ and $b\in R(w)$.
One says that $F$ is a~\emph{tree set}
if $G(w)$ is a tree for every $w\in F$.
If $G(w)$ is connected for every $w\in F$, then one
says that $F$ is \emph{connected}.

The class of uniformly recurrent tree sets contains the
extensively studied class of \emph{Arnoux-Rauzy sets}. 
The Arnoux-Rauzy sets over two letters, of which the Fibonacci set
is an example, are the well known \emph{Sturmian sets}.
The Tribonacci set is also an example of an Arnoux-Rauzy set.
See the survey \cite{Glen&Justin:2009} and the research paper~\cite{Berthe&Felice&Dolce&Leroy&Perrin&Reutenauer&Rindone:2015} (note
that in \cite{Berthe&Felice&Dolce&Leroy&Perrin&Reutenauer&Rindone:2015}
the Arnoux-Rauzy sets are called Sturmian).

\begin{Example}
  Here is an example, taken from~\cite{Dolce&Perrin:2018},  of a uniformly recurrent connected set
  which is not a tree set.
  Take the Tribonnaci set
  $F=F(\psi)$, as in Example~\ref{eg:tribonacci-set}.
  Consider the morphism $\alpha$
  given by $\alpha(a)=\alpha(b)=a$ and $\alpha(c)=c$.
  Since identifying letters clearly transforms
  connected extensions graphs into connected extension graphs,
  $\alpha(F)$ is connected, but, as observed
  in~\cite[Example 3.6]{Dolce&Perrin:2018},
  the extension graph of $a^3$ in $\alpha(F)$ is not a tree (it is the complete bipartite graph $K_{2,2}$).
\end{Example}

\subsection{Parses}

A \emph{parse} of a word $w$ with respect to a subset $X$ of $A^*$
is a triple~$(v,x,u)$
such that $w=vxu$ with $v\in A^*\setminus A^* X$, $x\in X^\ast$
and $u\in A^*\setminus X A^*$.
The number of parses of $w$ with respect to $X$ is denoted by~$\delta_X(w)$.

\begin{Rmk}\label{rmk:number-of-special-prefixes}
  For every set $X$,  the integer $\delta_X(w)$ is greater than or equal to the number of prefixes of $w$ not in $A^* X$.
  Equality holds when $X$ is a prefix code (cf.~\cite[Proposition 6.1.6]{Berstel&Perrin&Reutenauer:2010}).
\end{Rmk}

\begin{Lemma}[{cf.~\cite[Subsection 4.1]{Berstel&Felice&Perrin&Reutenauer&Rindone:2012}}]\label{l:parse-inequalities}
  If $X$ is a bifix code, then the inequality $\delta_X(v)\leq \delta_X(uvw)$
  holds, for every $u,v,w\in A^*$.
\end{Lemma}

Consider a factorial subset $F$ of $A^*$.
The \emph{$F$-degree} of $X$, denoted by $d_F(X)$,
is the supremum of the set $\{\delta_X(w)\mid w\in F\}$. It may be infinite.
The \emph{degree} of $X$ is $d_{A^*}(X)$, usually denoted by $d(X)$.

In this paper, we pay special attention
to bifix codes with finite $F$-degree, where $F$ is recurrent.
A characterization of such codes is given in the next theorem, taken from~\cite{Berstel&Felice&Perrin&Reutenauer&Rindone:2012}. For its statement, we need some more definitions. When $X,F\subseteq A^*$, one says that
$X$ is \emph{$F$-thin} if there is a word in $F$ that
is not a factor of an element of $X$; a \emph{thin set}
is an $A^*$-thin set. A word $u\in A^*$ is an
\emph{internal factor} of
a word $w\in A^*$ if $w\in A^+uA^+$.

\begin{Thm}[{cf.~\cite[Theorem 4.2.8]{Berstel&Felice&Perrin&Reutenauer&Rindone:2012}}]\label{t:characterization-of-finite-degree-codes}
  Let $F$ be a recurrent set and let $X$ be a bifix code contained in $F$.
  The $F$-degree of $X$ is finite
  if and only if $X$ is an $F$-thin and $F$-maximal bifix code.
  In this case, a word $w\in F$
  is an internal factor of a word of $X$
  if and only if $\delta_X(w)<d_F(X)$.
\end{Thm}

We turn our attention to rational bifix codes.

\begin{Prop}\label{p:drop-F-thin}
  For any recurrent set $F$,
  if $X$ is a rational bifix code contained in~$F$,
  then $X$ is $F$-thin.
\end{Prop}

\begin{proof}
  It suffices to show that there is a word in $X$ which is not a proper factor of
  a word of $X$. Suppose that that is not the case.
  Then, there is a sequence $(u_n)_n$ of words
  of $X$, with strictly increasing lengths, such that, for every $n\geq 1$,
  one has $u_{n+1}=x_nu_ny_n$ for
  some words $x_n,y_n$.
  Since $X$ is a rational prefix code,
  applying Proposition~3.2.9 from \cite{Berstel&Perrin&Reutenauer:2010},
  we conclude that
  there is $n_0$ such that $y_n=1$ for all $n\geq n_0$.
  And since $X$ is a bifix code, we get $u_n=u_{n_0}$ for all
  $n\geq n_0$, a contradiction.
\end{proof}

\begin{Rmk}\label{rmk:rational-codes-are-thin}
  Combining
  Theorem~\ref{t:characterization-of-finite-degree-codes}
  and Proposition~\ref{p:drop-F-thin}
  we conclude that
  a rational bifix code, contained in a recurrent
  set $F$, has finite $F$-degree if and only if it is an $F$-maximal bifix code.
  In particular, a rational bifix code
  has finite degree if and only if it is a maximal bifix code --- in fact, every rational code is thin~\cite[Proposition 2.5.20]{Berstel&Perrin&Reutenauer:2010}.
\end{Rmk}

We next state some properties of the intersection of
a (uniformly) recurrent set with a thin maximal bifix code (equivalently, a bifix code of finite degree, cf.~Theorem~\ref{t:characterization-of-finite-degree-codes}).

\begin{Thm}[{\cite[Theorems 4.2.11 and 4.4.3]{Berstel&Felice&Perrin&Reutenauer&Rindone:2012}}]\label{t:thin-code-intersects-F}
  Let $Z$ be a maximal bifix code of $A^*$
  with finite degree and let $F$ be a recurrent
  subset of $A^*$. The intersection $X=Z\cap F$ is an $F$-maximal bifix code.
  One has $d_F(X)\leq d(Z)$, with equality if $Z$ is finite. If, moreover, $F$ is uniformly recurrent, then $X$ is finite.
\end{Thm}

A \emph{group code} of $A^*$ is a subset $Z$ of $A^*$
such that $Z^\ast$ is recognized by a finite \emph{group automaton}, that is,
a trim
automaton whose initial state is the unique final state and such that the action
of each letter of $A$ on the finite set of states is a permutation.
Notice that the group automaton defining a group code $Z$ is
complete, deterministic and reduced, and so it is the minimal automaton of $Z^\ast$. Several properties
of group codes can be found in~\cite[Section 6]{Berstel&Felice&Perrin&Reutenauer&Rindone:2012}. For instance, a group code is a maximal bifix code with degree equal to the number of states of the group automaton defining~it.

The next result is a property which is explained in detail in the proof
of \cite[Theorem 5.10]{Berthe&Felice&Dolce&Leroy&Perrin&Reutenauer&Rindone:2015d}. It is a consequence of the
fact that uniformly recurrent tree sets satisfy the so called
\emph{finite index basis property}, cf.~\cite[Theorem 4.4]{Berthe&Felice&Dolce&Leroy&Perrin&Reutenauer&Rindone:2015c}. 

\begin{Thm}\label{t:equal-degree-tree-set-group-code}
  If $Z$ is a group code of $A^*$ and $F$
  is a uniformly recurrent tree set with alphabet $A$,
  then $d_F(Z\cap F)=d(Z)$.
\end{Thm}

In Section~\ref{sec:f-group-rational} a generalization
of Theorem~\ref{t:equal-degree-tree-set-group-code}
is deduced for arbitrary uniformly recurrent connected sets (cf.~Corollary~\ref{t:equal-degree-connected-set-group-code}),
and actually for a more general setting expressed in Theorem~\ref{t:group-code-isomor-maximal-subgroups}.
Note that the finite index basis property applies only for uniformly recurrent tree sets
(cf.~\cite[Corollary 4.11]{Berthe&Felice&Dolce&Leroy&Perrin&Reutenauer&Rindone:2015c}).

In contrast, we have the following example. It concerns
a uniformly recurrent set which is not connected, but
that belongs to the class of  \emph{eventually tree sets}
(studied in~\cite{Dolce&Perrin:2018}), meaning that the
extension graph  of every sufficiently long word is a tree.\footnote{In~\cite{Dolce&Perrin:2018}, and also in~\cite{Perrin:2018}, the subshifts corresponding to (eventually) tree sets are called~\emph{(eventually) dendric
    subshifts}.}

 \begin{Example}\label{eg:non-forbidden-example}
   Consider the group automaton over the alphabet $A=\{a,b,c,d\}$
   presented in Figure~\ref{fig:group-automaton}, with initial state $1$. 
   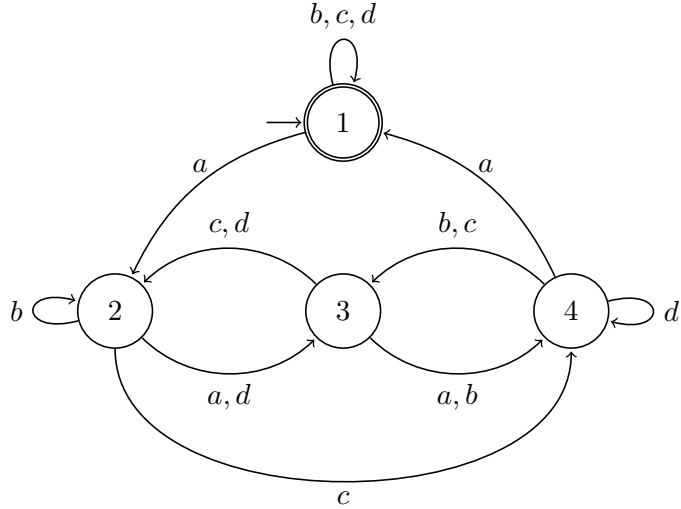
\begin{figure}[h]
     \centering
     \begin{tikzpicture}[shorten >=1pt, node distance=2.5cm and 3cm, on grid,initial text=,semithick]
  \node[state]   (2)                {$2$};
  \node[state]   (3) [right=of 2]   {$3$};
  \node[state]   (4) [right=of 3]   {$4$};
  \node[state,initial,accepting]   (1) [above=of 3]   {$1$};
  \path[->]  (3)   edge  [bend right=45] node [above] {$c,d$} (2)
             (3)   edge  [bend right=45]  node [below] {$a,b$} (4)
             (2)   edge  [bend right=45] node [below] {$a,d$} (3)
             (4)   edge  [bend right=45]  node [above] {$b,c$} (3)
             (1)   edge  [bend right=25]  node [above] {$a$} (2)
             (4)   edge  [bend right=25]  node [above] {$a$} (1)
             (2)   edge  [bend right=90]  node [below] {$c$} (4)
             (1)      edge [loop above]   node         {$b,c,d$} ()
             (2)      edge [loop left]   node         {$b$} ()
             (4)      edge [loop right]   node         {$d$} ();

\end{tikzpicture}
     \caption{A group automaton of degree $4$.}
     \label{fig:group-automaton}
   \end{figure}

   Let $Z$ be the group code defined by this group automaton.
   Consider the set $Y=\{ab,ac,bc,cd,ca,da\}$.
   If $u\in Y^\ast\cup Y^\ast A$,
   then every path in the automaton which is labeled by $u$, and starts at
   state $3$, passes only through states belonging to $\{2,3,4\}$. Hence, every element of $Y^\ast \cup Y^\ast A$
   is an internal factor of an element of $Z$.

    Let $\varphi$ be the substitution over $A$ given
    by
    \begin{equation*}
      \varphi(a)=ab,\quad \varphi(b)=cda,\quad \varphi(c)=cd
      \quad\text{and}\quad\varphi(d)=abc.
    \end{equation*}
    This substitution is borrowed from~\cite[Example 3.4]{Berthe&Felice&Dolce&Leroy&Perrin&Reutenauer&Rindone:2015}.
    Let $F$ be the uniformly recurrent subset
    of $A^*$ defined by $\varphi$.
    If $u\in F\setminus \{1\}$, then $G(u)$
    is a tree, but $G(1)$ is acyclic with
    two connected components, displayed in Figure~\ref{fig:extension-graph}.
    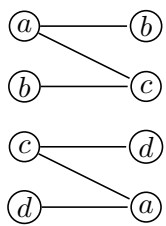
\begin{figure}[h]
      \centering
     \begin{tikzpicture}[shorten >=1pt, node distance=0.8cm and 1.6cm, on grid,initial text=,semithick]
       \tikzstyle{state}=[draw,circle,minimum size=1em,inner sep=1]
  \node[state]   (1)                {$a$};
  \node[state]   (2) [below=of 1]   {$b$};
  \node[state]   (3) [right=of 1]   {$b$};
  \node[state]   (4) [below=of 3]   {$c$};
  \path      (1)   edge (3)
             (1)   edge (4)
             (2)   edge (4);
  \node[state]   (5) [below=of 2]   {$c$};
  \node[state]   (6) [below=of 5]   {$d$};
  \node[state]   (7) [right=of 5]   {$d$};
  \node[state]   (8) [below=of 7]   {$a$};
  \path      (5)   edge (7)
             (5)   edge (8)
             (6)   edge (8);           
     \end{tikzpicture}
     \caption{The extension graph $G(1)$, with $L(1)$ on the left column and
     $R(1)$ on the right column.}
     \label{fig:extension-graph}
    \end{figure}
    Notice that $F\subseteq Y^\ast\cup Y^\ast A$,
    and so every element of $F$ is an internal factor of an element of $Z$.
    Hence, we have $d_F(Z\cap F)<d(Z)$.
 \end{Example}

\subsection{Restriction to rational languages}
 
 In this paper we are concerned with intersections of
 the form $X=Z\cap F$, in which
 $Z$ is a maximal bifix code of $A^*$
 and $F$ is a recurrent subset of~$A^*$, for some finite alphabet $A$.
 Our results restrict to the case where
 both $Z$ and $X$ are rational. The rationality of
 $Z$ implies that of $X$ when the recurrent set $F$ is
 in one of the following two situations: $F$ is rational, or $F$ is uniformly recurrent, and in the second case $X$ is even finite, as seen in Theorem~\ref{t:thin-code-intersects-F} (cf.~Remark~\ref{rmk:rational-codes-are-thin}).
 From the viewpoint of symbolic dynamics, this corresponds to the
 two most studied classes of symbolic dynamical systems: sofic systems (corresponding to rational recurrent sets) and minimal systems (corresponding to uniformly recurrent sets). Out of this realm, it is possible to
 have~$Z$ rational and $X$ not rational,
 as seen in the next example.

 \begin{Example}\label{eg:intersection-not-rational}
   Consider the group code $Z$ whose minimal automaton
   is represented in Figure~\ref{fig:intersection-dyck},
   with initial state $1$.
   \begin{figure}[h]
     \centering
     \begin{tikzpicture}[shorten >=1pt, node distance=3cm, on grid,initial text=,semithick]
  \node[state,initial above,accepting]   (1)                {$1$};
  \node[state]   (2) [right=of 1]   {$2$};
  \node[state]   (3) [right=of 2]   {$3$};
  \path[->]  (1)   edge  [loop left]     node [left]  {$a,b$}
             (1)   edge  [bend right=30] node [below] {$c,d$} (2)
             (2)   edge  [bend right=30] node [above] {$c,d$} (1)
             (2)   edge  [bend right=30] node [below] {$a,b$} (3)
             (3)   edge  [bend right=30] node [above] {$a,b$} (2)
             (3)   edge  [loop right]    node [right] {$c,d$} (3);
\end{tikzpicture}
     \caption{A group automaton of degree $3$.}
     \label{fig:intersection-dyck}
   \end{figure}
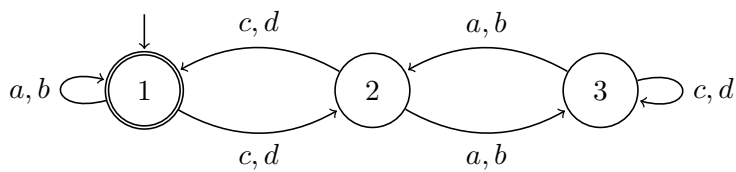
   Suppose that $a$, $b$, $c$ and $d$
   represent respectively the four parentheses
   $``{(}"$, $``{)}"$, $``{[}"$ and $``{]}"$,
   and that $F$ is the recurrent set of factors of words of $\{a,b,c,d\}^\ast$
   that are correctly parenthesized. The set $F$
   defines a \emph{Dyck shift} (a class
   of symbolic dynamical systems going back to~\cite{Krieger:1974}).
   For every positive integer~$n$, the word $ca^nb^nd$ belongs
   to~$X=Z\cap F$. If $X$ were rational, then, by the Pumping Lemma,
   we would have $ca^nb^{n+k}d$
   for some positive integers $n$ and $k$. But such a word
   does not belong to $F$, whence $X$ is not rational.
 \end{Example}

\subsection{The free profinite monoid}

We introduce free profinite monoids,
giving the introductory texts~\cite{Almeida:2003cshort,Pin:2009S}
and the books \cite{Almeida:1994a,Rhodes&Steinberg:2009qt}
as supporting references.
Let $A$ be a finite alphabet.
When $u$ and $v$ are distinct elements of $A^*$, there is
some finite monoid $M$ and some homomorphism $\varphi\colon A^*\to M$
such that $\varphi(u)\neq \varphi(v)$.
Denote by $r(u,v)$ the least possible value for the cardinal of $M$.
The function $d\colon A^*\times A^*\to \mathbb R^+$
such that
\begin{equation*}
  d(u,v)=
  \begin{cases}
    2^{-r(u,v)}&\text{if $u\neq v$}\\
     0 &\text{if $u=v$}
  \end{cases}
\end{equation*}
is a metric. It is in fact an ultrametric: $d(u,v)\leq\max\{d(u,w),d(w,v)\}$,
for all $u,v,w\in A^*$.
The metric space~$A^*$ thus defined admits a completion~$\F A$, which is actually a compact metric space.
The multiplication in $A^*$ is uniformly continuous with respect to $d$,
and so $\F A$ is a topological monoid whose 
multiplication is the unique continuous extension
of the multiplication of~$A^*$.
In particular, $A^*$ is a dense submonoid of $\F A$.
Intuitively, the elements of~$\F A$ can be viewed as generalizations of
words over $A$. The elements of $A^*$ are isolated in the topological space $\F A$. We say that the elements of $\F A$ are \emph{pseudowords} over the alphabet $A$. The elements of $\F A\setminus A^*$ are the \emph{infinite pseudowords}, while those of $A^*$ are the \emph{finite pseudowords}.
The topological monoid $\F A$ is a special
example of a profinite monoid. A \emph{profinite monoid} is a compact\footnote{We adopt the convention that being compact requires being Hausdorff.} monoid $M$ such that, if $u$ and $v$ are distinct elements of $M$, then $\varphi(u)\neq\varphi(v)$ for some continuous homomorphism
$\varphi\colon M\to N$, where $N$ is a finite monoid endowed with the discrete topology. Note that the finite monoids are profinite, if endowed with the discrete topology, as we do from hereon. It turns out that $\F A$ is the \emph{free profinite monoid generated by $A$}, in the following sense: if $\varphi\colon A\to N$ is a mapping into a profinite monoid, then there is a unique extension of $\varphi$
to a continuous homomorphism $\hat\varphi\colon \F A\to N$.

The notion of profinite monoid can be generalized to abstract topological algebras. In particular, we can consider
\emph{profinite semigroups} and \emph{profinite groups}, and the latter will appear frequently along this paper.
The \emph{free profinite semigroup}
generated by a finite alphabet $A$ is denoted by $\widehat {A^+}$. The construction of $\widehat {A^+}$ is entirely similar to that of $\F A$. Moreover, the closed subsemigroup $\F A\setminus \{1\}$ can be identified with $\widehat {A^+}$.
Frequently, results concerning $\widehat {A^+}$
have an immediate translation to $\F A$, and vice-versa.
The reader should keep this in mind when checking references.
Finally, we denote by $F_G(A)$
the free group generated by $A$,  and
by $\FG A$ the \emph{free profinite group} generated by $A$, which
has $F_G(A)$ as a dense subgroup.
The cardinal of $A$ is the \emph{rank} of $\FG A$.
These notions can be generalized to infinite alphabets, but some care is needed when doing that. Except for mentioning \emph{en passant} an example,
we shall not need to consider such generalizations.

Viewing $\FG A$ as a profinite monoid, we may
consider the canonical projection from $\F A$ onto $\FG A$,
which is the unique continuous homomorphism $p_G\colon \F A\to \FG A$
fixing the elements of $A$.

In this paper we explore some connections between structural aspects
of the free profinite monoid and the structure of the syntactic monoid (whose definition we recall in Subsection~\ref{sec:syntactic-monoid}) of the monoid generated by a code.
That is why we next recall some basics of the structural theory
of monoids~\cite{Clifford&Preston:1961,Almeida:1994a,Rhodes&Steinberg:2009qt}.
Green's relations $\mathcal J$, $\mathcal R$ and $\mathcal L$
are defined by
\begin{equation*}
  u\mathrel{\mathcal J} v\Leftrightarrow MuM=MvM,\quad
  u\mathrel{\mathcal R} v\Leftrightarrow uM=vM,\quad
  u\mathrel{\mathcal L} v\Leftrightarrow Mu=Mv.
\end{equation*}
The other important Green's relations are
$\mathcal D=\mathcal R\vee\mathcal L$
and $\mathcal H=\mathcal R\cap\mathcal L$.
An element $u$ of $M$ is \emph{regular} if $u\in uMu$.
A $\mathcal D$-class is \emph{regular} if it contains a regular element, equivalently, if all its elements are regular.
The regular $\mathcal D$-classes are the ones that contain idempotents.
In the monoid $M$, the $\mathcal H$-classes of idempotents are the
maximal subgroups of $M$, with respect to the inclusion relation (a \emph{subgroup} of $M$ is a subsemigroup of $M$ which is a group). The maximal subgroups contained in the same regular $\mathcal D$-class are isomorphic. 
If $M$ is profinite, then $\mathcal J=\mathcal D$,
and the maximal subgroups of the same regular $\mathcal D$-class
are isomorphic as profinite groups.
Furthermore, we have the following
elementary property, isolated for future reference.

\begin{Prop}\label{p:maximal-onto-maximal}
  Consider a continuous homomorphism
  $\varphi\colon M\to N$ of profinite
  monoids.  Let $H$ and $H'$ be maximal subgroups
  of $M$ contained in the same $\mathcal D$-class.
  Then
  $\varphi(H)$ and $\varphi(H')$ are isomorphic profinite groups.
  Moreover, $\varphi(H)$ is a maximal subgroup of $N$ if and only if
  $\varphi(H')$ is a maximal subgroup of $N$.
\end{Prop}

\begin{proof}
  The proposition is a direct consequence of
  Green's Lemma (cf.~\cite[Appendix A]{Rhodes&Steinberg:2009qt}), which in this setting
  affirms in particular the existence of elements $x,y\in M$
  such that $h\in H\mapsto xhy\in H'$
  is a continuous isomorphism, with
  inverse given by $h\in H'\mapsto x'hy'\in H$
  for some $x',y'\in M$.
\end{proof}

The abstract (profinite) group
defining the isomorphism class of the maximal subgroups of
a regular $\mathcal D$-class of a (profinite) monoid
is the \emph{Sch\"utzenberger group} of the $\mathcal D$-class.
This notion can be extended to non-regular $\mathcal D$-classes, but
we shall not need the generalization for this paper.

A profinite monoid $M$ also satisfies the \emph{stability property}, that
states that $u\mathrel{\mathcal J}ux$
if and only if $u\mathrel{\mathcal R}ux$,
and that $u\mathrel{\mathcal J}xu$
if and only if $u\mathrel{\mathcal L}xu$,
for every $u,x\in M$. We shall occasionally use this property without reference.

We denote by $\overline Y$ the topological closure in $\F A$ of a subset $Y$
of $\F A$. We have the following characterization of rational languages, which
offers a glimpse of the usefulness of the free profinite monoid.

\begin{Thm}[{cf.~\cite[Theorem 3.6.1]{Almeida:1994a}}]\label{t:closure-of-rational-lang}
  Let $L$ be a language of $A^*$.
  Then $\overline{L}$ is open if and only if
  $L$ is rational.
\end{Thm}

For a proof of the following useful technical result,
see \cite[Section 3]{Almeida&ACosta&Costa&Zeitoun:2017}.
It is an improved version of~\cite[Lemma 2.5]{Almeida&ACosta:2007a}.
It expresses in terms of sequences the property that the multiplication in $\F A$ is an open mapping.

\begin{Lemma}\label{l:refine-factorization}
  Let $u,v,w\in \F A$ be such that $w=uv$.
  Suppose that $(w_n)_n$ is a sequence of elements of $\F A$
  converging to $w$.
  Then there are factorizations $w_n=u_nv_n$
  such that the sequences $(u_n)_n$ and $(v_n)_n$ respectively converge
  to $u$ and $v$.
\end{Lemma}

If $S$ is a compact semigroup, and $s\in S$, then
the closed subsemigroup $\overline{s^+}$
of $S$ generated by $s$ contains a unique idempotent, denoted $s^\omega$.
It is the neutral element of the unique maximal subgroup $K_s$
of the compact semigroup~$\overline{s^+}$. Moreover,
$K_s$ is the minimum ideal of $\overline{s^+}$.
The element $s\cdot s^\omega$ is denoted $s^{\omega+1}$,
and its inverse in $K_s$ is denoted by $s^{\omega-1}$.
If $S$ is profinite, then $s^\omega=\lim s^{n!}$. 

\subsection{The topological closure of a recurrent set}

If $F$ is a factorial subset of $A^*$,
then $\overline F$ is itself factorial, that is, it contains its factors in
$\F A$, and  if $F$ is recurrent, then $\overline F$
also satisfies the property that $u,v\in\overline F$ implies the existence of $w\in \overline{F}$ such that $uwv\in\overline F$~\cite{Almeida&ACosta:2007a}.
Using a standard compactness argument (see \cite[Proposition 3.6]{ACosta&Steinberg:2011}), one also sees that when $F$ is a recurrent subset of $A^*$, there is a regular $\mathcal J$-class $J(F)$ contained in $\overline F$
such that $\overline{F}$ is the set of factors of elements of $J(F)$.
It is the unique $\mathcal J$-class with these properties, and it is the
minimum $\mathcal J$-class which has nonempty intersection with $\overline {F}$.
In the uniformly recurrent case, $J(F)$ occupies a very special place
in the structure of $\F A$, as seen next.

\begin{Thm}[{\cite{Almeida:2004a}}]
  An element $u$ of $\F A\setminus A^*$
  is such that all of its proper factors are elements of $A^*$
  if and only if $u\in J(F)$ for some uniformly recurrent set $F$.
\end{Thm}

We denote by $G(F)$ the Sch\"utzenberger group of $J(F)$, when $F$ is recurrent.
It is shown in~\cite{Almeida&Volkov:2006} that if $F$ is
periodic, then $G(F)$
is a free profinite group of rank $1$,
while in~\cite{ACosta&Steinberg:2011}
it is proved that if $F$ is a non-periodic rational recurrent subset of~$A^*$, then
$G(F)$ is a free profinite group of rank $\aleph_0$.
Concerning uniformly recurrent sets, the identification of
$G(F)$ has only been made for some special classes.
For instance, if $F$ is a uniformly recurrent tree set with alphabet $A$,
then $G(F)$ is a free profinite group with rank $|A|$,
and the restriction to a maximal subgroup of $J(F)$
of the canonical projection from $\F A$ onto $\FG A$
is a continuous isomorphism~\cite[Section 6]{Almeida&ACosta:2016b}.
 More generally, we have
 the next weaker property
 for all uniformly recurrent connected sets,
 a property which we shall invoke later in this paper.

 \begin{Thm}\label{t:projection-maximal-subgroup-connected-set}
   Let $F$ be a uniformly recurrent connected subset of $A^*$
   with alphabet $A$. 
   If $H$ is a maximal subgroup of $J(F)$
   and $p_G$ is the canonical projection from $\F A$ onto $\FG A$,
   then $p_G(H)=\FG A$.
 \end{Thm}

 In the next few paragraphs we give a proof of Theorem~\ref{t:projection-maximal-subgroup-connected-set}. Like the proof for the above mentioned stronger property
 for uniformly recurrent tree sets deduced in~\cite[Section 6]{Almeida&ACosta:2016b}, it uses the important concept of return word.
 If $F$ is a recurrent set and $u\in F$, then
 the set of \emph{first return words of $F$ to $u$} is the
 prefix code $R_F(u)$ consisting of the words $v$ such that $uv\in F\cap A^\ast u$
 and $u$ is not an internal factor of $uv$.
 The set $F$ is uniformly recurrent if and only if
 $R_F(u)$ is finite for every $u\in F$.
 The following theorem from~\cite{Berthe&Felice&Dolce&Leroy&Perrin&Reutenauer&Rindone:2015} is one of the main ingredients
 in the proof of Theorem~\ref{t:projection-maximal-subgroup-connected-set}.
 The alphabet of a factorial set $F$ of words is the set of letters
 belonging to $F$.
 
 \begin{Thm}\label{t:return-sets-generate-free-groups}
   If $F$ is a uniformly recurrent connected set with alphabet~$A$, then $R_F(u)$ is a generating set of the free group
   over~$A$, for every $u\in F$.
 \end{Thm}

 The following tool, encapsulated
 in~\cite[Lemma 5.3]{Almeida&ACosta:2016b},
 is the other main ingredient in the proof of Theorem~\ref{t:projection-maximal-subgroup-connected-set}.

 \begin{Lemma}\label{l:technical-lemma-on-approximation-of-H}
   Let $F$ be a uniformly non-periodic recurrent subset of $A^*$. 
   If $H$ is a maximal subgroup of $J(F)$,
   then there is a sequence $(X_n)_{n\geq 1}$ of
   finite codes of $A^*$
   satisfying
     $H=\bigcap_{n\ge 1}\overline{X_n^+}$,
   with
      $X_1^+
      \supseteq X_2^+
      \supseteq X_3^+
      \supseteq
      \cdots
      $,
      and such that, for each $n\geq 1$, there is $u_n\in F$
      for which the set $X_n$ is conjugated
   in $F_G(A)$ with~$R_F(u_n)$.
 \end{Lemma}
 
 We are now ready to prove Theorem~\ref{t:projection-maximal-subgroup-connected-set}.

 \begin{proof}[Proof of Theorem~\ref{t:projection-maximal-subgroup-connected-set}]
   If $F$ is a periodic set, then
   $R_F(u)$ is a singleton for all
   words $u$ in $F$ with sufficiently large length,
   and so, accordingly to Theorem~\ref{t:return-sets-generate-free-groups},
   one has that $A=\{a\}$ is a one-letter alphabet
   and $F=a^*$.
   Hence, in that case, $H$ is the set $\F A\setminus A^\ast$,
   which projects via $p_G$ onto $\FG A$.

   Suppose that $F$ is not a periodic set. Let $(X_n)_n$ be a sequence
   of codes as in Lemma~\ref{l:technical-lemma-on-approximation-of-H}.
   Then, by the continuity of the homomorphism $p_G$,
   we know that,
   for every $n\geq 1$, the set $p_G\bigl(\overline{X_n^+}\bigr)$
   is the closed subgroup of $\FG A$ generated by $X_n$.
   Using the hypothesis that $X_n$ is conjugated in $F_G (A)$ with
   $R_F(u_n)$, for some $u_n\in F$, and applying
   Theorem~\ref{t:return-sets-generate-free-groups},
   we deduce that $p_G\bigl(\overline{X_n^+}\bigr)=\FG A$,
   for every $n\geq 1$.
   Since the sequence of sets $(\overline{X_n^+})_n$
   forms a descending chain for the inclusion,
   using a standard compactness argument we conclude that
   the image of the intersection
   $\bigcap_{n\geq 1}\overline{X_n^+}$ by $p_G$ is $\FG A$,
   that is, $p_G(H)=\FG A $.
 \end{proof}

 \subsection{Stable and unitary sets}
 \label{sec:unitary-sets}

 A submonoid $N$ of a monoid $M$
  is called \emph{stable}\index{stable}
 if for all $u,v,w\in M$, whenever $u,vw,uv,w\in N$
 we have $v\in N$. It is called \emph{right unitary}
 if for every $u,v\in A^*$, $u,uv\in N$ implies $v\in N$.
 There is also the dual notion of \emph{left unitary} code.
 It is well known (cf.~\cite[Propositions 2.2.5 and 2.2.7]{Berstel&Perrin&Reutenauer:2010}) that a submonoid of $A^*$ is stable
 if and only if it is generated by a code,
 and it is right unitary (respectively, left unitary)
 if and only if it is generated by a prefix code (respectively, by a suffix code).

  \begin{Prop}\label{p:propositionUnitary}
    Let $N$ be a recognizable submonoid of $A^*$,
    and consider its topological closure
    $\overline{N}$ in $\F A$.
    The following implications are true:
\begin{enumerate}
\item if $N$ is stable, then $\overline{N}$ is stable;
\item if $N$ is right unitary, then $\overline{N}$ is right unitary;
\item if $N$ is left unitary, then $\overline{N}$ is left unitary.
\end{enumerate}
\end{Prop}

\begin{proof}
The set $\overline{N}$ is clearly a submonoid of $\widehat{A^*}$.

Assume that $N$ is stable.
Let $u,v,w\in \F A$ be such that $uv,w,u,vw\in \overline{N}$.
Let $(u_n),(v_n)$ and $(w_n)$ be  sequences of
words  converging   to $u,v$ and $w$, respectively,
with $u_n,w_n\in N$. 
Since, by Theorem~\ref{t:closure-of-rational-lang},
the set $\overline{N}$ is open in $\F A$, we have $u_nv_n,v_nw_n\in\overline{N}$ for all large enough $n$, and thus also
$u_nv_n,v_nw_n\in N$ for all large enough $n$ (indeed, $N=\overline{N}\cap A^*$, as the elements
of $A^*$ are isolated in $\F A$).
Because $N$ is stable, this implies that $v_n\in N$
for all large enough $n$, which in turn implies that $v\in\overline{N}$.
Thus $\overline{N}$ is stable.

The proofs of the second and third implications
are similar.
\end{proof}
 
 \subsection{The syntactic monoid}
 \label{sec:syntactic-monoid}

Fix a finite alphabet $A$.
For a language $L$ of~$A^*$,
and for $u\in A^*$,
let
\begin{equation*}
C_L(u)=\{(x,y)\in A^*\times A^*:xuy\in L\}  
\end{equation*}
be the context of $u$ in $L$.
Let $\Syn(L)$ be the syntactic monoid of $L$
and let $\eta_L\colon A^*\to \Syn(L)$ be
the corresponding syntactic homomorphism.
Recall that $\eta_L(u)=\eta_L(v)$
if and only if $C_L(u)=C_L(v)$.

Sometimes, one considers the \emph{syntactic order} in $\Syn(L)$, which is defined by
$\eta_L(u)\leq \eta_L(v)$ if $C_L(u)\subseteq C_L(v)$.
This notion goes back to Sch\"utzenberger~\cite{Schutzenberger:1956b}, and was rediscovered in~\cite{Pin:1995a}.\footnote{The reader is warned that
in~\cite{Pin:1995a}, and some other later papers, it is adopted the order which is the reverse of ours and of Sch\"utzenberger's, and also of
the more recent paper~\cite{Almeida&Cano&Klima&Pin:2015}.}
The syntactic order is compatible with multiplication, in
the sense that
$\eta_L(u)\leq \eta_L(v)$
implies
$\eta_L(wu)\leq \eta_L(wv)$
and
$\eta_L(uw)\leq \eta_L(vw)$
for every $w\in A^*$.
This makes $\Syn(L)$ an \emph{ordered monoid},
a monoid endowed with a partial order compatible with multiplication
(see~\cite{Pin:1995a} for details).
We remark that in a finite ordered monoid the restriction of the partial
order to a subgroup is always the identity (cf.~\cite[Lemma 6.4]{Pin:1996b}).

In case $L$ is a rational language, that is, in case
$\Syn(L)$ is finite, we may also consider the unique
continuous homomorphism $\hat\eta_L$ from $\F A$ to $\Syn(L)$ extending 
$\eta_L$. This leads to the consideration of the set
\begin{equation*}
  C_{\overline L}(u)=\{(x,y)\in\F A:xuy\in\overline{L}\}.
\end{equation*}

The following lemma, proved in~\cite{ACosta:2006},
will be used several times and without reference. 

\begin{Lemma}\label{l:syntactic-inequality-profinite-version}
  Suppose $L$ is a rational language.
  Then $\hat\eta_L(u)\leq \hat\eta_L(v)$ if and only if
  $C_{\overline L}(u)\subseteq C_{\overline L}(v)$.
\end{Lemma}

In what follows we consider
deterministic trim automata, not necessarily complete.
Recall that the syntactic monoid of $L$ is (isomorphic to) the transition monoid of the minimal deterministic automaton $\mathcal M_L$ of $L$. For this reason, when $L$ is rational, if $p$ and $q$ are states of $\mathcal M_L$, and
if $u$ is an element of $\F A$, we can write
$p\cdot u=q$ with the meaning that $p\cdot \hat\eta_L(u)=q$.
In such a setting, it is straightforward to see that
if $i$ is the initial state of $\mathcal M_L$ and $t$ is one of its final states, then, for every $u\in\F A$, one has $i\cdot \hat\eta_L(u)=t$ 
if and only if $u\in\overline{L}$.
In the proof of the following proposition, we use this fact,
as well as the fact that if $X$ is a prefix code, then the initial state $i_X$
of $\mathcal M_{X^\ast}$ is the unique final state of $\mathcal M_{X^\ast}$.

\begin{Prop}\label{p:idempotents-are-in-the-closure-of-Xast}
  Let $F$ be a recurrent subset of $A^*$.
  Let the subset $X$ of~$F$ be a rational $F$-maximal bifix code of $A^*$. If $e$ is an idempotent belonging to~$\overline{F}$, then
  $e\in\overline{X^\ast}$.
\end{Prop}

\begin{proof}
  Recall that $X$ is $F$-thin, by Proposition~\ref{p:drop-F-thin}.
  Let $(w_n)_n$ be a sequence of elements of $F$ converging to $e$.
    It is shown in~\cite[Theorem 4.2.2]{Berstel&Felice&Perrin&Reutenauer&Rindone:2012}
    that an $F$-thin and $F$-maximal bifix code is left $F$-complete,
    whence there is $v_n\in A^*$ such that $v_nw_n\in X^\ast$.
  Therefore, if $v$ is an accumulation point of the sequence
  $(v_n)_n$, we have $ve\in \overline{X^\ast}$,
  and so in $\mathcal M_{X^\ast}$
  we have $i_X\cdot ve=i_X$.
  As $e$ is idempotent, it follows that $i_X\cdot e=i_X$,
  establishing that $e\in\overline{X^\ast}$.
\end{proof}

Consider a rational code $X$ of $A^*$.
Let $J(X)$ denote\footnote{Notice that $X$ is never recurrent, and so there is no risk of confusion with the notation $J(F)$ when $F$ is a recurrent set.} the minimum ideal of $\Syn(X^\ast)$, and let $G(X)$ be the Sch\"utzenberger group of
$J(X)$. One says that $G(X)$ is the \emph{group of $X$}.
This notion is introduced in~\cite{Berstel&Perrin&Reutenauer:2010} for the larger class of the so called \emph{very thin codes}, but we restrict to the simpler setting of rational codes, as our main results concern them. The reader should have in mind
this when checking the references. 
Additionally, consider a recurrent subset $F$
of~$A^*$. We shall denote by $J_F(X)$ the regular $\mathcal J$-class containing $\hat\eta_{X^\ast}(J(F))$, and by $G_F(X)$ the Sch\"utzenberger group
of $J_F(X)$. One says that $G_F(X)$ is the \emph{$F$-group of $X$}.
Using a simple continuity argument one sees that $J_F(X)$
is the minimum $\mathcal J$-class of $\Syn(X^\ast)$ that
has nonempty intersection with $\eta_{X^\ast}(F)$.
For that reason, we may say that $J_F(X)$ is the \emph{$F$-minimum $\mathcal J$-class} of $\Syn(X^*)$.

Recall that
in a monoid of (partial) transformations,
the images of the transformations in the same $\mathcal J$-class
have the same cardinal,
the common cardinal being the \emph{rank} of
the $\mathcal J$-class.
Viewing the syntactic monoid $M(L)$ as a partial transformation
monoid acting in the minimal automaton $\mathcal M_L$,
the rank of a word $u$ in $\mathcal M_L$ is the rank of $\eta_L(u)$.
The $F$-degree of a rational $F$-maximal bifix code
is closely related with the rank
of the elements of $J_F(X)$, as seen next (for which the reader should have in mind Remark~\ref{rmk:rational-codes-are-thin}).

\begin{Prop}[{\cite[cf.~Proposition 7.1.3 and Lemma 7.1.4]{Berstel&Felice&Perrin&Reutenauer&Rindone:2012}}]\label{p:d-parses-implies-you-are-in-JFX}
  Let $F$ be a recurrent subset and let $X$
  be a rational $F$-maximal bifix code.
  Take $u\in F$. Then $u$ has rank $d_F(X)$ in the minimal automaton
  of~$X^\ast$ if and only if $\eta_{X^\ast}(u)\in J_F(X)$.
  If $\delta_X(u)=d_F(X)$, then $u$ has rank $d_F(X)$.
\end{Prop}

The converse of the last implication in
Proposition~\ref{p:d-parses-implies-you-are-in-JFX}
fails: if $Z$ is a group code of $A^*$ of degree $d$,
then every word $u$ of $A^*$ has rank $d$ in the minimal automaton of~$Z^\ast$.

It is also well known that in a transformation monoid, a subgroup
contained in a $\mathcal J$-class of rank $n$ embeds in $S_n$. We shall revisit
this fact in Section~\ref{eg:not-H-equivalent}. Meanwhile, we register
the following corollary of Proposition~\ref{p:d-parses-implies-you-are-in-JFX}.

\begin{Cor}\label{c:upper-bound-size-FGroups}
   Let $F$ be a recurrent subset and let $X$
   be a rational $F$-maximal bifix code.
   Then, we have $G_F(X)\leq S_{d_F(X)}$.
   In particular, if $Z$ is a rational maximal
   bifix code then $G(Z)\leq S_{d(Z)}$.
\end{Cor}

\section{Parses of a pseudoword}

The generalization to pseudowords of the notion of parse was introduced
in~\cite{Kyriakoglou&Perrin:2017}. In this section, we show that for rational sets with finite $F$-degree, the number of parses is continuous, if we endow $\mathbb N$ with
the discrete topology. This property and a direct consequence of
it are used in Section~\ref{sec:f-group-rational}.

\subsection{The number of parses of a pseudoword is continuous}

Let $X$ be a subset of $A^*$.
A parse of a pseudoword  $w$ with respect
to $X$ is a triple $(v,x,u)$
such that $w=vxu$ with $v\in \F A\setminus \overline{A^* X}$, $x\in \overline{X^\ast}$
and $u\in \F A\setminus \overline{X A^*}$.
If $w\in A^*$, then this is precisely the notion of parse
we already gave for elements of $A^*$,
because for every language $L$ of $A^*$ we have $\overline{L}\cap A^*=L$.
Let $\delta_X(w)$ be the number of parses of $w$, whenever $w\in\F A$.

We begin with a couple of preparatory technical lemmas.

\begin{Lemma}\label{l:explosion-on-the-number-of-parses}
  Let $X$ be a rational subset of $A^*$, and let $w\in \F A$.
  Suppose that~$v$ is a prefix of $w$
  such that $v\in\F A\setminus\overline{A^* X}$
  and such that $v=vt$ for some $t\in \F A$ with $t\neq 1$.
  Consider a positive integer $\ell$
  and let $(w_n)_n$ be a sequence of elements of $A^*$
  converging to $w$.
  There is $n_0$ such that if $n>n_0$ then~$\delta_{X}(w_n)>\ell$.
\end{Lemma}

\begin{proof}
  Let $u\in\F A$ be such that $w=vu$.
  Since $w=vt^\ell u$, and in view of Lemma~\ref{l:refine-factorization}, we can consider factorizations 
  \begin{equation*}
    w_n=v_nt_{n,1}t_{n,2}\cdots t_{n,\ell} u_n
  \end{equation*}
  such that the sequences
  $(v_n)_n$, $(t_{n,i})_n$  and $(u_n)_n$ respectively converge
  to $v$, $t$ and $u$,
  where $i$ is any element of $\{1,\ldots,\ell\}$.
  As $t\neq 1$, for all large enough $n$ and every $i\in \{1,\ldots,\ell\}$,
  we have $t_{n,i}\neq 1$.
  For each $j\in \{0,1,\ldots,\ell\}$,
  let $v_{n,j}=v_nt_{n,1}t_{n,2}\cdots t_{n,j}$, with $v_{n,0}=v_n$.
  Notice that $\lim v_{n,j}=vt^j=v$.
  Since $X$ is rational, the set
  $\overline{A^*\setminus A^* X}=\F A\setminus\overline{A^* X}$
  is an open neighborhood of $v$, and so $v_{n,j}\in A^*\setminus A^* X$
  for all large enough $n$ and every $j\in \{0,1,\ldots,\ell\}$.
  Because $t_{n,j}\neq 1$ for every $j\in \{1,\ldots,\ell\}$,
  the set
  $\{v_{n,0},v_{n,1},\ldots,v_{n,\ell}\}$ has $\ell+1$ elements,
  all of them in $A^*\setminus A^* X$, and all
  of them prefixes of $w_n$, for all large enough $n$.
  This implies that $\delta_{X}(w_n)\geq \ell+1$,
  for all large enough~$n$.
\end{proof}

\begin{Lemma}\label{l:continuity-of-the-parsing-function}
  Consider a rational subset $X$ of $A^*$.
  Let $w\in\F A$. Suppose that $(w_n)_n$ is a sequence of elements of $A^*$ converging
  to $w$.
  For every $d\in\mathbb N$,
  if the set
  $\{\delta_X(w_n):n\in\mathbb N\}$
  is bounded and the set
  $\{n\in\mathbb N:\delta_X(w_n)\geq d\}$
  is infinite, then the inequality $\delta_X(w)\geq d$ holds.
  Conversely, if $\delta_X(w)\geq d$,
  then $\delta_X(w_n)\geq d$
  for all large enough $n$.
  Moreover,  if the integer $d$ satisfies $\delta_X(w)=d$,
  then $\delta_X(w_n)=d$
  for all large enough $n$.
\end{Lemma}

\begin{proof}
  Suppose that
  the set
  $\{\delta_X(w_n):n\in\mathbb N\}$
  is bounded and that
  the set $\{n\in\mathbb N:\delta_X(w_n)\geq d\}$
  is infinite.
  Clearly, by taking subsequences, we are reduced to
  consider the case
  where this set is precisely~$\mathbb N$.
  For each~$n$, and
  with respect to $X$,
  consider $d$ distinct parses
  $p_{n,i}=(\alpha_{n,i},\beta_{n,i},\gamma_{n,i})$ of $w_n$, where $i\in\{1,\ldots, d\}$.
  By compactness, there is a strictly increasing
  sequence of integers
  $(n_k)_k$
  such that the sequence
  \begin{equation}\label{eq:continuity-of-the-parsing-function}
  (p_{n_k,1},p_{n_k,2},\ldots,p_{n_k,d})_k  
  \end{equation}
  converges in $(\F A\times\F A\times\F A)^d$ to
  a $d$-tuple
  $(p_1,p_2,\ldots,p_d)$.
  For each $i\in\{1,\ldots,d\}$,
  let $p_i=(\alpha_i,\beta_i,\gamma_i)$.
  Notice that, since the sets
  $\overline{A^*\setminus A^* X}$,  $\overline{X^\ast}$,
  and $\overline{A^*\setminus XA^*}$ are closed,
  the triple $p_i$ is a parse of $w$.

  Let $i,j$ be distinct elements of $\{1,\ldots,d\}$.
  For each $n$, there is $t_n\in A^*$
  such that
  $\alpha_{n,j}=\alpha_{n,i}t_n$
  or such that $\alpha_{n,i}=\alpha_{n,j}t_n$.
  
  Suppose that $t_{n_k}\neq 1$ for infinitely many $k$, and let
  $t$ be an accumulation point of $(t_{n_k})_k$. Then
  $t\neq 1$, and at least one of the equalities
  $\alpha_j=\alpha_it$
  or
  $\alpha_i=\alpha_jt$
  holds.
  By Lemma~\ref{l:explosion-on-the-number-of-parses},
  and since
  $\{\delta_X(w_n):n\in\mathbb N\}$
  is bounded, we must have $\alpha_j\neq\alpha_i$, whence
  $p_i\neq p_j$.

  On the other hand, if $t_{n_k}=1$, that is, if
  $\alpha_{n_k,i}=\alpha_{n_k,j}$, then,
  because $p_{n_k,i}\neq p_{n_k,j}$,
  we must have
  $\gamma_{n_k,i}\neq \gamma_{n_k,j}$.
  Therefore,
  if $t_{n_k}\neq 1$ for only finitely many~$k$,
  then for infinitely many $k$ we have either
  $\gamma_{n_k,i}\in A^+\gamma_{n_k,j}$
  or
  $\gamma_{n_k,j}\in A^+\gamma_{n_k,i}$.
  By compactness, this implies the existence of some
  $s\in\F A\setminus\{1\}$
  such that
  $\gamma_i=s\gamma_j$
  or
  $\gamma_j=s\gamma_i$.
  Applying the dual of Lemma~\ref{l:explosion-on-the-number-of-parses},
  we deduce that $\gamma_i\neq\gamma_j$, thus $p_i\neq p_j$.

  All cases considered, we conclude that $p_i\neq p_j$
  whenever $i\neq j$. Therefore, we have $\delta_X(w)\geq d$.

  Conversely, suppose that $\delta_X(w)\geq d$.
  Consider a set
  \begin{equation*}
    \bigl\{(v_i,x_i,u_i):i\in\{1,\ldots,d\}\bigr\}
  \end{equation*}
  of $d$ parses of $w$ with respect to $X$.
  By Lemma~\ref{l:refine-factorization}, we can consider factorizations
  \begin{equation*}
    w_n=v_{n,i}x_{n,i}u_{n,i}
  \end{equation*}
  such that $\lim_{n\to\infty}(v_{n,i},x_{n,i},u_{n,i})=(v_i,x_i,u_i)$.
  In particular, for all large enough $n$,
  the set
  \begin{equation}\label{eq:many-parses-2}
    \bigl\{(v_{n,i},x_{n,i},u_{n,i}):i\in\{1,\ldots,d\}\bigr\}
  \end{equation}
  has $d$ elements.
  Since the sets
  $\overline{A^*\setminus A^* X}$,
  $\overline{X^\ast}$
  and $\overline{A^*\setminus A^* X}$
  are open,
  for all large enough $n$
  we have $v_{n,i}\in A^*\setminus{A^* X}$,
  $x_{n,i}\in X^\ast$ and $u_{n,i}\in A^*\setminus{XA^*}$.
  Therefore, the elements of~\eqref{eq:many-parses-2}
  are distinct parses of $w_n$, thus $\delta_X(w_n)\geq d$, for all large enough $n$.
  
  Suppose moreover that $\delta_X(w)=d$.
  Then $\delta_X(w_n)\geq d$ for all large enough~$n$. If
  the set $\{n\in\mathbb N:\delta_X(w_n)\geq d+1\}$
  were infinite, then by what was already proved, we would have $\delta_X(w)\geq d+1$. Therefore, we must have $\delta_X(w_n)=d$ for all large enough $n$.
\end{proof}

We are now ready for the main result of this section.

\begin{Prop}\label{p:the-parse-is-continuous}
  Consider a factorial set $F$ of $A^*$.
  Let $X$ be a rational subset of $F$ with finite $F$-degree $d$.
  Then $\delta_X(w)\leq d$ for every $w\in\overline{F}$,
  and the mapping $\delta_X\colon \overline{F}\to\{1,\ldots,d\}$
  thus defined is continuous when we endow $\{1,\ldots,d\}$
  with the discrete topology.
\end{Prop}

\begin{proof}
  Let $w\in\overline{F}$.
  Suppose that $(w_n)_n$
  is a sequence of elements of $F$ converging to $w$.
  We first claim that $\delta_X(w)=\delta_X(w_n)$ for all sufficiently large $n$.  
  For each $i\in\mathbb N$, let
  \begin{equation*}
    W_i=\{n\in\mathbb N:\delta_X(w_n)=i\}.  
  \end{equation*}
  Since $X$ has finite $F$-degree $d$,
  the set
  \begin{equation*}
    I=\{i\in \mathbb N:W_i\text{ is infinite}\}
  \end{equation*}
  is nonempty and its maximum $M$ is at most $d$.
  By Lemma~\ref{l:continuity-of-the-parsing-function},
  we know that $\delta_X(w)\geq M$.
  If we had $\delta_X(w)\geq M+1$ then,
  again by Lemma~\ref{l:continuity-of-the-parsing-function},
  the set $W_{M+1}$ would be infinite, contradicting the maximality of $M$.
  Therefore, we have $\delta_X(w)=M$.
  By Lemma~\ref{l:continuity-of-the-parsing-function},
  this implies that $\delta_X(w_n)=M$ for all large enough $n$, proving the claim.

  Let $k\in \{1,\ldots,d\}$. Consider a sequence $(u_n)_n$ of elements
  $\F A$ converging to $u$ and such that $\delta_X(u_n)=k$ for every $n$.
  By what was proved in the previous paragraph, for each $n$
  we can find $v_n\in A^*$ such that $d(u_n,v_n)<\frac{1}{n}$
  and $\delta_X(v_n)=k$.
  As
  \begin{equation*}
  d(v_n,u)\leq \max\{d(v_n,u_n),d(u_n,u)\}\xrightarrow[n\to\infty]{}0,
  \end{equation*}
  we have $\lim v_n=u$, and so $\delta_X(u)=k$, also by
  the case considered in the first paragraph.
  Therefore, $\delta_X^{-1}(k)$ is closed. This shows that
  $\delta_X$ is continuous if we endow $\{1,\ldots,d\}$
  with the discrete topology.
\end{proof}

\begin{Rmk}\label{rmk:profinite-version-number-of-special-prefixes}
  Let $X$ be a subset of $A^+$,
  and let $w\in\F A$. Denote by $\tilde{\delta}_X(w)$
  the number of prefixes of $w$ not in $\overline{A^* X}$.
  In the statements of Lemmas~\ref{l:explosion-on-the-number-of-parses} and~\ref{l:continuity-of-the-parsing-function}, and of
  Proposition~\ref{p:the-parse-is-continuous}, if we replace all occurrences of $\delta_X$ by $\tilde{\delta}_X$,
  then we still have true statements, as is easily seen after a
  straightforward adaptation of the respective proofs. In particular, in view of Remark~\ref{rmk:number-of-special-prefixes}, if $X$ is a rational prefix code of finite $F$-degree,
  then $\delta_X(w)=\tilde{\delta}_X(w)$
  for every $w\in\F A$.
  We shall not use these facts.
\end{Rmk}

The following lemma is an easy consequence
of Proposition~\ref{p:the-parse-is-continuous}
that will be used later on.

\begin{Lemma}\label{l:degree-in-JF}
  Consider a recurrent subset $F$ of $A^*$,
  Let $X$ be a rational bifix code of~$A^*$ with finite $F$-degree.
    For every $u\in J(F)$, we have $\delta_X(u)=d_F(X)$.    
\end{Lemma}

\begin{proof}
  Let $u\in J(F)$ and $v\in F$ be such that $\delta_X(v)=d_F(X)$.
  There is a sequence $(u_n)_n$ of elements of $F$ converging to $u$
  such that $v$ is a factor of $u_n$, for every $n$.
  From Lemma~\ref{l:parse-inequalities}
  and the maximality of $\delta_X(v)$,
  we deduce that $\delta_X(v)=\delta_X(u_n)$. Applying Proposition~\ref{p:the-parse-is-continuous}, we conclude that $\delta_X(u)=d_F(X)$.
\end{proof}

\subsection{The number of parses of $\mathcal H$-equivalent pseudowords}

We conclude this section with some observations that are not applied in the rest of the paper, but that the reader may find interesting. 

Let $v\in\F A$.
Say that $v$ is \emph{left-cancelable} if the following
property holds: $v=vt$ implies $t=1$.
Because $\F A$ is equidivisible~(proved in~\cite{Almeida&ACosta:2007a},
see~\cite{Almeida&ACosta:2017}),
the pseudoword $v$ is  left-cancelable if and only if,
for every $x,y\in\F A$, the equality $vx=vy$ implies $x=y$. Therefore,
if $v$ is left-cancelable and $u\in v\cdot\F A$, then there is a unique $w\in\F A$ such that $u=vw$.
We denote such $w$ by $v^{-1}u$. There is a dual notion of \emph{right-cancelable}
pseudoword $v$, for which, whenever $u\in\F A\cdot v$, we use the notation $uv^{-1}$
for the unique element $w$ of $\F A$ such that $u=wv$.

\begin{Prop}\label{p:how-to-count-parses}
  Consider a factorial subset $F$ of $A^*$.
  Let $X$ be a rational subset
  of $A^*$ with finite $F$-degree.
  For every $u\in\overline{F}$, the number $\delta_X(u)$ is equal to the number
  of pairs $(s,p)$ of pseudowords such that, for some pseudoword~$x$, the triple
  $(s,x,p)$ is a parse of $u$ with respect to $X$.
  When $(s,p)$ is such a pair, $s$ is left-cancelable
  and $p$ is right-cancelable.
\end{Prop}

\begin{proof}
  Suppose that $(s,x,p)$ and $(s,y,p)$ are parses of $u$ with respect to~$X$.
  By Lemma~\ref{l:explosion-on-the-number-of-parses},
  if $s=st$ for some $t\neq 1$, then
  there are words in $F$ whose image under $\delta_X$ is arbitrarily large,
  contradicting the assumption that
  $X$ has finite $F$-degree.
  Hence, $s$ is left-cancelable. Dually, $p$ is right-cancelable.
  As $sxp=syp$, it follows that $x=y$. Hence, a parse
  of an element of $\overline{F}$ with respect to $X$ is determined
  by the first and third component.
\end{proof}

The following result establishes a relationship between $\delta_X$
and $\eta_{X^\ast}$.

\begin{Prop}\label{p:link-between-syntactic-homomorphism-and-number-of-parses}
  Consider a factorial subset $F$ of $A^*$.
  Let $X$ be a rational
  subset of $A^*$ with finite $F$-degree.
  Suppose that the pseudowords
  $u$ and $v$ of $\F A$ are
  such that $u\in\overline{F}$,  $u=euf$ and $v=evf$,
  with $e,f$ idempotents and $e\mathrel{\mathcal R}u\mathrel{\mathcal L}f$.
  If $\hat\eta_{X^\ast}(u)\leq\hat\eta_{X^\ast}(v)$ holds, then we have
  $\delta_X(u)\leq\delta_X(v)$.
\end{Prop}

\begin{proof}
  Let $(s,w,p)$ be a parse of $u$ with respect to~$X$.  
  Then $s$ is a prefix of $e$ and $p$ is a suffix of $f$.
  By
  Proposition~\ref{p:how-to-count-parses},
  $s$ is left-cancelable, and $p$ is right-cancelable.
  Therefore, we may consider the pseudowords $s^{-1}e$ and $ep^{-1}$.
  Moreover, as $s(s^{-1}e)u(fp^{-1})p=euf=swp$,
  we have $w=(s^{-1}e)u(fp^{-1})$. As $w$ belongs to $\overline{X^\ast}$,
  and $\hat\eta_{X^\ast}(u)\leq\hat\eta_{X^\ast}(v)$,
  it follows that the pseudoword $(s^{-1}e)v(fp^{-1})$ also belongs to $\overline{X^\ast}$.
  Hence, $(s,(s^{-1}e)v(fp^{-1}),p)$ is a parse of $v$ with respect to $X$.
  By Proposition~\ref{p:how-to-count-parses},
  we conclude that $\delta_X(u)\leq\delta_X(v)$.
\end{proof}

\begin{Cor}
    Consider a factorial subset $F$ of $A^*$.
  Let $X$ be a rational
  subset of $A^*$ having finite $F$-degree.
  Suppose that $u$ and $v$ are
  $\mathcal H$-equivalent with elements of
  $\overline{F}$.
  If $\hat\eta_{X^\ast}(u)=\hat\eta_{X^\ast}(v)$ holds, then we have
   $\delta_X(u)=\delta_X(v)$.
\end{Cor}

\begin{proof}
  If $u$ and $v$ are not regular,
   then $u=v$ by~\cite[Corollary 13.2]{Rhodes&Steinberg:2001}.
   Assuming that $u$ and $v$ are $\mathcal H$-equivalent regular
   pseudowords,
   we can consider idempotents $e$ and $f$
   such that $u$ and $v$ are $\mathcal R$-equivalent to $e$
   and $\mathcal L$-equivalent to $f$.
   We may then apply
   Proposition~\ref{p:link-between-syntactic-homomorphism-and-number-of-parses}
   to establish the result.
\end{proof}

\section{Main results}
\label{sec:f-group-rational}

For a subset $Y$ of $\F A$, say that a pseudoword
$u$ is \emph{forbidden in $Y$} if $u$ is not a factor of an element of $Y$.
For the next proposition, recall that, if $Z$ is a rational and maximal bifix code, then $d(Z)$ is finite (cf.~Remark~\ref{rmk:rational-codes-are-thin}).

\begin{Prop}\label{p:equal-degree-imlies-forbidden}
  Let $Z$ be a rational maximal bifix code of $A^*$.
  Suppose that $F$ is a recurrent subset of $A^*$
  and that the intersection $X=Z\cap F$ is rational.
  The equality $d_F(X)=d(Z)$ holds if and only if the elements of $J(F)$
  are forbidden in $\overline{Z}$.
  Moreover, if $d_F(X)=d(Z)$
  then $\hat\eta_{Z^\ast}(J(F))\subseteq J(Z)$.
\end{Prop}

\begin{proof}
  Suppose that $d_F(X)=d(Z)$.
  Let $u\in J(F)$. Suppose that $u$ is a factor of an element $z$ of $\overline{Z}$. Take $w\in F$. Then there are $a,b\in A$
  such that $awb$ is a factor of $u$, and thus of $z$.
  Therefore, since $\overline{A^* awb A^*}$ is open in $\F A$,
  the word $awb$ is a factor of an element of $Z$.
  This implies that $\delta_Z(w)<d(Z)$,
  by Theorem~\ref{t:characterization-of-finite-degree-codes}.
  Since one clearly has $\delta_X(w)=\delta_Z(w)$,
  it follows that $d_F(X)<d(Z)$, contradicting
  that $d_F(X)=d(Z)$.
  Therefore, if $d_F(X)=d(Z)$, then
  every element of $J(F)$  is forbidden in $\overline{Z}$.

  Conversely, suppose that the elements of $J(F)$ are forbidden in $\overline{Z}$. Let $u\in J(F)$.
  Consider a sequence $(u_n)_n$ of elements of $F$ converging
  to $u$. For all sufficiently large $n$,
  we know that $u_n$
  is not an internal factor of some element of $Z$.
  Hence, by Theorem~\ref{t:characterization-of-finite-degree-codes},
  we have $\delta_Z(u_n)=d(Z)$, for all sufficiently large~$n$.
  Notice that, as $u_n\in F$, we may write $\delta_Z(u_n)=\delta_X(u_n)$.
  Applying Proposition~\ref{p:the-parse-is-continuous},
  we conclude that $\delta_X(u)=d(Z)$.
  On the other hand, by Lemma~\ref{l:degree-in-JF},
  we have $\delta_X(u)=d_F(X)$. This establishes the converse implication
  of the equivalence in the statement of the proposition.

  Finally, since $\delta_Z(u_n)=d(Z)$,
  we have $\hat\eta_{Z^\ast}(u_n)\in J(Z)$
  by Proposition~\ref{p:d-parses-implies-you-are-in-JFX},
  thus $\hat\eta_{Z^\ast}(u)\in J(Z)$.  
\end{proof}

The notion of forbidden pseudoword in a set was introduced
having in mind the following key proposition.

\begin{Prop}\label{p:fundamental-syntactic-inequalities}
  Let $Z$ and $F$ be subsets of $A^*$,
  with $F$ factorial, and let $X=Z\cap F$.
  Suppose, moreover, that $Z^\ast$ and $X^\ast$
  are rational.
  Let $e,f$ be idempotents of $\F A$, and let $u\in\overline{F}$ and $v\in\F A$ be such that $u=euf$
  and $v=evf$. The following implications hold:
  \begin{enumerate}
  \item $\hat\eta_{X^\ast}(u)\leq \hat\eta_{X^\ast}(v)\Rightarrow \hat\eta_{Z^\ast}(u)\leq \hat\eta_{Z^\ast}(v)$, if $e$ and $f$ are forbidden in $\overline{Z}$\label{item:fundamental-syntactic-inequalities-1};
  \item $\hat\eta_{Z^\ast}(v)\leq \hat\eta_{Z^\ast}(u)
    \Rightarrow \hat\eta_{X^\ast}(v)\leq \hat\eta_{X^\ast}(u)$, if $e$ and $f$ are forbidden in $\overline{X}$\label{item:fundamental-syntactic-inequalities-2}.
  \end{enumerate}
\end{Prop}

\begin{proof}
  Suppose that $\hat\eta_{X^\ast}(u)\leq \hat\eta_{X^\ast}(v)$.
  Consider pseudowords $\alpha,\beta\in\F A$ such that \mbox{$\alpha u\beta\in \overline{Z^\ast}$}.
  Because $u=euf$,
  and in view of Lemma~\ref{l:refine-factorization}, we may consider sequences $(\alpha_n)$, $(\beta_n)$, $(e_n)$, $(f_n)$, $(u_n)$ and $(v_n)$, of words in $A^*$, with $e_nu_nf_n\in F$,
  respectively converging to $\alpha$, $\beta$, $e$, $f$, $u$ and $v$.  
  It follows that $\lim\alpha_ne_nu_nf_n\beta_n=\alpha u\beta\in \overline{Z^\ast}$.
  Because $Z^\ast$ is rational,
  we may as well assume that $\alpha_ne_nu_nf_n\beta_n\in Z^\ast$ for all $n$.
  Since, by hypothesis, the pseudowords $e$ and $f$ are not factors
  of elements of $\overline{Z}$,
  there is $n_0$ such that for all $n>n_0$
  there are factorizations $e_n=e_{n,1}e_{n,2}$
  and $f_n=f_{n,1}f_{n,2}$
  for which the words $\alpha_ne_{n,1}$, $e_{n,2}u_nf_{n,1}$,
  and $f_{n,2}\beta_n$ all belong to $Z^\ast$.
  Hence, as $e_{n,2}u_nf_{n,1}\in F$, we
  have $e_{n,2}u_nf_{n,1}\in X^*$, when $n>n_0$.
  Let $(e_1,e_2,f_1,f_2)$ be an accumulation point of
  the tuple $(e_{n,1},e_{n,2},f_{n,1},f_{n,2})$.
  We then have $e_2uf_1\in\overline{X^\ast}$ and
  $\alpha e_1,f_2\beta\in \overline{Z^\ast}$.
  Applying the hypothesis $\hat\eta_{X^\ast}(u)\leq\hat\eta_{X^\ast}(v)$,
  we obtain $e_2vf_1\in\overline{X^\ast}$.
  Therefore, we have
  \begin{equation*}
    \alpha v\beta= \alpha e_1\cdot e_2 v f_1\cdot f_2\beta\in
    \overline{Z^\ast}\cdot \overline{X^\ast}\cdot \overline{Z^\ast}\subseteq \overline{Z^\ast}, 
  \end{equation*}
  thereby establishing that $\hat\eta_{Z^\ast}(u)\leq \hat\eta_{Z^\ast}(v)$.

  Conversely, suppose that $\hat\eta_{Z^\ast}(v)\leq \hat\eta_{Z^\ast}(u)$.
  Let $\alpha,\beta\in\F A$ be such that $\alpha v\beta\in \overline{X^\ast}$.
  We may consider sequences $(\alpha_n)$, $(\beta_n)$, $(e_n)$, $(f_n)$, $(u_n)$ and $(v_n)$, of words in $A^*$, with $e_nu_nf_n\in F$,
  respectively converging to $\alpha$, $\beta$, $e$,
  $f$, $u$ and~$v$. 
  We have
  $\lim\alpha_ne_nv_nf_n\beta_n=\alpha v\beta\in \overline{X^\ast}$.
  Since $X^\ast$ is rational,
  we may as well assume that $\alpha_ne_nv_nf_n\beta_n\in X^\ast$ for all $n$.
  As $e,f$ are forbidden in $\overline{X}$,
  we conclude that, for some $n_0$, and every $n>n_0$,
  there are factorizations $e_n=e_{n,1}e_{n,2}$
  and $f_n=f_{n,1}f_{n,2}$
  such that the words $\alpha_ne_{n,1}$, $e_{n,2}v_nf_{n,1}$,
  and $f_{n,2}\beta_n$ all belong to $X^\ast$.
   Let $(e_1,e_2,f_1,f_2)$ be an accumulation point of
  the tuple $(e_{n,1},e_{n,2},f_{n,1},f_{n,2})$.
  We then have $e_2vf_1\in\overline{X^\ast}$ and
  $\alpha e_1,f_2\beta\in \overline{X^\ast}$.
  Applying the hypothesis $\hat\eta_{Z^\ast}(v)\leq \hat\eta_{Z^\ast}(u)$,
  we obtain $e_2uf_1\in\overline{Z^\ast}$.
  But, since $e_2uf_1\in\overline{F}$
  and $Z^\ast$ is rational,
  we actually have $e_2uf_1\in\overline{X^\ast}$. 
  We conclude that
  $\alpha u\beta= \alpha e_1\cdot e_2 u
  f_1\cdot f_2\beta
  \in\overline{X^\ast}$.
\end{proof}

\begin{Cor}\label{c:isomorphism-between-images}
  Under the assumptions of
  Proposition~\ref{p:fundamental-syntactic-inequalities},
  if $H$ is a maximal subgroup of $J(F)$
  whose elements are forbidden in $\overline{X}$,
  then the group $\hat\eta_{X^\ast}(H)$
  is a homomorphic image of the group $\hat\eta_{Z^\ast}(H)$.
  If, moreover, the elements of $H$ are forbidden in $\overline{Z}$,
  then $\hat\eta_{X^\ast}(H)$ and $\hat\eta_{Z^\ast}(H)$
  are isomorphic. 
\end{Cor}

\begin{proof}
  By Proposition~\ref{p:fundamental-syntactic-inequalities}\eqref{item:fundamental-syntactic-inequalities-2}, if the elements of $H$ are forbidden
  in $\overline{X}$, then the correspondence
  \begin{equation*}
  \hat\eta_{Z^\ast}(u)\mapsto \hat\eta_{X^\ast}(u)\quad (u\in H),
  \end{equation*}
  is a well-defined homomorphism from $\hat\eta_{Z^\ast}(H)$
  onto $\hat\eta_{X^\ast}(H)$.
  By Proposition~\ref{p:fundamental-syntactic-inequalities}\eqref{item:fundamental-syntactic-inequalities-1},
  this homomorphism is an isomorphism if the elements of $H$ are forbidden in $\overline{Z}$.
\end{proof}

Under additional assumptions, we next obtain a closer relationship
between the homomorphisms $\hat\eta_{X^\ast}$
and $\hat\eta_{Z^\ast}$ appearing in Proposition~\ref{p:fundamental-syntactic-inequalities}.

\begin{Prop}\label{p:well-defined-homor-images}
  Let $F$ be a factorial subset of $A^*$.
  Suppose that $Z$
  is a rational prefix code of $A^*$ and that
  $X=Z\cap F$ is a rational $F$-maximal prefix code.
  Consider idempotents $e,f$ that are forbidden
  in~$\overline{Z}$.
  Let $u\in\overline{F}$ and $v\in\F A$ be such that $u=euf$
  and $v=evf$.
  If we have $\hat\eta_{X^\ast}(u)\leq\hat\eta_{X^\ast}(v)$, then
  the equality $\hat\eta_{Z^\ast}(u)=\hat\eta_{Z^\ast}(v)$ holds.
\end{Prop}

\begin{proof}
  Suppose that the inequality
  $\hat\eta_{X^\ast}(u)\leq \hat\eta_{X^\ast}(v)$ holds.
  By Proposition~\ref{p:fundamental-syntactic-inequalities}\eqref{item:fundamental-syntactic-inequalities-1},
  we have
  $\hat\eta_{Z^\ast}(u)\leq \hat\eta_{Z^\ast}(v)$.
  We proceed to prove that $\hat\eta_{Z^\ast}(v)\leq \hat\eta_{Z^\ast}(u)$.
  Let $\alpha,\beta\in\F A$ be such that
  $\alpha v\beta\in\overline{Z^\ast}$.
  Since $e$ and $f$ are forbidden in $\overline{Z}$ and by Lemma~\ref{l:refine-factorization}, there are factorizations $e=e_1e_2$ and $f=f_1f_2$
  such that
  the pseudowords $\alpha e_1$,
  $e_2vf_1$ and $f_2\beta$ belong to $\overline{Z^\ast}$.
  Because $f$ is idempotent and $v=vf$,
  we have $f=(ff_1)f_2$ and $e_2vf_1=e_2v(ff_1)$,
  and so we are reduced to the case where $f_1=ff_1$, which we suppose
  from hereon to hold.
  Notice that the pseudoword $e_2uf_1$ is a factor of $u$,
  and thus it belongs to $\overline{F}$.
  Let $(w_n)_n$ be a sequence of elements
  of $F$ converging to $e_2uf_1$. Thanks to Lemma~\ref{l:refine-factorization},
  there is a factorization $w_n=e_{2,n}u_nf_{1,n}$
  such that $\lim e_{2,n}=e$, $\lim u_n=u$ and $\lim f_{1,n}=f_1$.
  By Proposition~\ref{p:right-F-complete}, every element of $F$ is a prefix of an element of $X^\ast$.
  On the other hand, $f_1=ff_1$ is not a factor of an element of
  $\overline{X}$, as $f$ itself is not. Hence,
  taking subsequences, we may suppose that
  there is a factorization $f_{1,n}=t_ns_n$
  such that $e_{2,n}u_nt_n$ is an element of $X^\ast$
  and $s_n$ is a proper prefix of an element of $X$.
  Let $t$ and $s$ be respectively an accumulation point of the
  sequence $(t_n)_n$ and of the sequence $(s_n)_n$.
  Note that
  \begin{equation}\label{eq:well-defined-homor-images-1}
    e_2ut\in\overline{X^\ast}.
  \end{equation}
  Applying the hypothesis 
  $\hat\eta_{X^\ast}(u)\leq\hat\eta_{X^\ast}(v)$,
  we get
  \begin{equation}\label{eq:well-defined-homor-images-2}
   e_2vt\in\overline{X^\ast}.
  \end{equation}
  On the other hand, by the definition of $e_2$ and $f_1$, we have
  \begin{equation}\label{eq:well-defined-homor-images-3}
      e_2vt\cdot s=e_2vf_1\in\overline{Z^\ast}.
  \end{equation}
  Since $\overline{Z^\ast}$ is right unitary
  (cf.~Proposition~\ref{p:propositionUnitary})
  it follows from
  \eqref{eq:well-defined-homor-images-2}
  and
  \eqref{eq:well-defined-homor-images-3}
  that $s\in \overline{Z^\ast}$.
  This means that $s_n\in Z^\ast$
  for sufficiently large $n$.
  Because $s_n$ is a proper prefix of an element of $X\subseteq Z$
  and $Z$ is a prefix code, we deduce that $s_n=1$
  for sufficiently large $n$, thus $s=1$ and $t=f_1$.
  Therefore, in view of \eqref{eq:well-defined-homor-images-1},
  we conclude that $\alpha u\beta =\alpha e_1\cdot e_2uf_1\cdot f_2\beta
  \in\overline{Z^\ast}$,
  thereby establishing that $\hat\eta_{Z^\ast}(v)\leq \hat\eta_{Z^\ast}(u)$.
\end{proof}

We now apply the preceding tools to deduce relationships between the maximal subgroups of
$J(F)$, $J(Z)$ and $J_F(Z\cap F)$, for suitable $Z$ and $F$.

\begin{Thm}\label{t:injective-homomorphism-maximal-subgroups}
  Let $F$ be a factorial subset of $A^*$.
  Suppose that $Z$
  is a rational prefix code of $A^*$ and that
  $X=Z\cap F$ is a rational $F$-maximal prefix code.
  Let $H$ be a maximal subgroup of $\F A$
  contained in $\overline {F}$.
  Suppose also that the elements of $H$ are forbidden in $\overline{Z}$.
  Consider the maximal subgroup $H_X$
  of $\Syn(X^\ast)$ containing $\hat\eta_{X^\ast}(H)$
  and the maximal subgroup $H_Z$ of $\Syn(Z^\ast)$
  containing $\hat\eta_{Z^\ast}(H)$.
  There is an injective homomorphism $\alpha\colon H_X\to H_Z$
  such that the diagram
    \begin{equation}\label{eq:injective-homomorphism-maximal-subgroups-1}
    \xymatrix{
      H\ar[r]^{\hat\eta_{Z^\ast}}\ar[d]_{\hat\eta_{X^\ast}}&H_Z\\
      H_X\ar[ur]_\alpha&
    }
  \end{equation}
  commutes.
\end{Thm}

\begin{proof}
  Denote by $e$ the idempotent of $H$.
  Let $x\in e\cdot \F A\cdot e$
  be such that $\hat\eta_{X^\ast}(x)\in H_X$.
  Then $\hat\eta_{X^\ast}(e)=\hat\eta_{X^\ast}(x^\omega)$ holds.
  Applying Proposition~\ref{p:well-defined-homor-images}\eqref{item:fundamental-syntactic-inequalities-1}, 
  we conclude that $\hat\eta_{Z^\ast}(e)=\hat\eta_{Z^\ast}(x^\omega)$.
  Since $x=exe$, it follows that \mbox{$\hat\eta_{Z^\ast}(x)\in H_Z$}.

  Suppose that $\hat\eta_{X^\ast}(x)=\hat\eta_{X^\ast}(y)\in H_X$,
  with $x,y$ both in $e\cdot \F A\cdot e$.
  As argued in the previous paragraph,
  we know that $\hat\eta_{Z^\ast}(x)$ and
  $\hat\eta_{Z^\ast}(y)$ belong to~$H_Z$.
  On the other hand,
  again by Proposition~\ref{p:well-defined-homor-images}\eqref{item:fundamental-syntactic-inequalities-1},
  from $\hat\eta_{X^\ast}(e)=\hat\eta_{X^\ast}(yx^{\omega-1})$
  we deduce $\hat\eta_{Z^\ast}(e)=\hat\eta_{Z^\ast}(yx^{\omega-1})$.
  Hence, the equality $\hat\eta_{Z^\ast}(x)=\hat\eta_{Z^\ast}(y)$ holds.

  Notice that $H_X$ is the image by $\hat\eta_{X^\ast}$ of the
  closed submonoid $M=\hat\eta_{X^\ast}^{-1}(H_X)\cap e\cdot \F A\cdot e$.
  In view of the previous two paragraphs,  we
  conclude that there is a unique homomorphism of groups
  $\alpha\colon H_X\to H_Z$ such that the diagram
      \begin{equation}\label{eq:injective-homomorphism-maximal-subgroups-2}
    \xymatrix{
      M\ar[r]^{\hat\eta_{Z^\ast}}\ar[d]_{\hat\eta_{X^\ast}}&H_Z\\
      H_X\ar[ur]_\alpha&
    }
  \end{equation}
  commutes.
  Suppose that
  $x\in M$
  is such that 
  $\hat\eta_{Z^\ast}(x)=\hat\eta_{Z^\ast}(e)$.
  Applying
  Proposition~\ref{p:fundamental-syntactic-inequalities}\eqref{item:fundamental-syntactic-inequalities-2},
  we obtain
  $\hat\eta_{X^\ast}(x)\leq \hat\eta_{X^\ast}(e)$.
  Since $\hat\eta_{X^\ast}(x)$ and $\hat\eta_{X^\ast}(e)$
  belong to the same maximal subgroup of $\Syn(X^\ast)$,
  it follows that $\hat\eta_{X^\ast}(x)=\hat\eta_{X^\ast}(e)$.
  Therefore, the homomorphism $\alpha$ is injective.
  It is a homomorphism for which clearly
  Diagram~\eqref{eq:injective-homomorphism-maximal-subgroups-1}
  commutes, since Diagram~\eqref{eq:injective-homomorphism-maximal-subgroups-2}
  commutes and $H\subseteq M$.
\end{proof}

Since we want to compare maximal subgroups
of the form $G(Z)$ and $G_F(X)$,
we are led by Theorem~\ref{t:injective-homomorphism-maximal-subgroups}
to the following couple of definitions.
Let $F$ be a recurrent subset of $A^*$.
Say that a rational code $Z$ of $A^*$
is \emph{$F$-charged} if $\hat\eta_{Z^*}$ maps
all (equivalently, some) maximal subgroups of $J(F)$ onto
maximal subgroups of $J(Z)$ (cf.~Proposition~\ref{p:maximal-onto-maximal}).
We also say that a rational code $X$ contained in $F$
is \emph{weakly $F$-charged} if $\hat\eta_{X^*}$ maps
all (equivalently, some) maximal subgroups of $J(F)$ onto
maximal subgroups of $J_F(X)$.
The next proposition gives a vast class of examples of $F$-charged codes.

\begin{Prop}\label{p:group-codes-are-connected-charged}
  Let $Z$ be a group code of $A^*$. If $F$
  is a uniformly recurrent connected set with alphabet $A$,
  then $Z$ is $F$-charged.
\end{Prop}

\begin{proof}
  Let $p_G$ be the canonical projection from $\F A$ onto $\FG A$.
  Since $\Syn(Z^\ast)$ is a finite group,
  we may consider the unique continuous homomorphism
  $\bar\eta_{Z^\ast}$ from the free profinite group $\FG A$ onto
  $\Syn(Z^\ast)$ such that $\bar\eta_{Z^\ast}\circ p_G=\hat \eta_{Z^\ast}$.
  Thanks to Theorem~\ref{t:projection-maximal-subgroup-connected-set},
  we have
  \begin{equation*}
    \hat\eta_{Z^\ast}(H)
  =\bar\eta_{Z^\ast}(p_G(H))
  =\bar\eta_{Z^\ast}(\FG A)=\Syn(Z^\ast)=G(Z),
\end{equation*}
establishing the proposition.
\end{proof}

With the definitions of $F$-charged and weakly $F$-charged code on hand, we state the following corollary of
Theorem~\ref{t:injective-homomorphism-maximal-subgroups}.

\begin{Cor}\label{c:injective-homomorphism-maximal-subgroups}
  Let $F$ be a factorial subset of $A^*$.
  Suppose that $Z$
  is a rational prefix code of $A^*$ and that
  $X=Z\cap F$ is a rational $F$-maximal prefix code.
  Suppose also that the elements of $J(F)$ are forbidden in $\overline{Z}$.
  The following conditions are equivalent:
  \begin{enumerate}
  \item $Z$ is $F$-charged;\label{item:injective-homomorphism-maximal-subgroups-1}
  \item $G_F(X)\simeq G(Z)$ and
    $X$ is weakly $F$-charged;\label{item:injective-homomorphism-maximal-subgroups-2}
  \item $|G_F(X)|=|G(Z)|$ and
    $X$ is weakly $F$-charged.\label{item:injective-homomorphism-maximal-subgroups-3}
  \end{enumerate}
\end{Cor}

\begin{proof}
  We retain the notation used in Theorem~\ref{t:injective-homomorphism-maximal-subgroups}.
  If we have the equality $\hat\eta_{Z^\ast}(H)=H_Z$, then the homomorphism
  $\alpha$ in Diagram~\eqref{eq:injective-homomorphism-maximal-subgroups-1}
  is onto. As
  Theorem~\ref{t:injective-homomorphism-maximal-subgroups} also asserts that
  $\alpha$ is injective,
  we conclude
  that~\eqref{item:injective-homomorphism-maximal-subgroups-1}$\Rightarrow$\eqref{item:injective-homomorphism-maximal-subgroups-2}
  holds.

  The implication~\eqref{item:injective-homomorphism-maximal-subgroups-2}$\Rightarrow$\eqref{item:injective-homomorphism-maximal-subgroups-3}
  is trivial.

  Finally, suppose that
  $|G_F(X)|=|G(Z)|$, that is, that $|H_Z|=|H_X|$.
  Then the injective homomorphism~$\alpha$
  is onto, because $H_Z$ is finite.
  Therefore, if  $|H_Z|=|H_X|$
  and $\hat\eta_{X^\ast}(H)=H_X$, we must have $\hat\eta_{Z^\ast}(H)=H_Z$
  by the commutativity of Diagram~\eqref{eq:injective-homomorphism-maximal-subgroups-1}, thus \eqref{item:injective-homomorphism-maximal-subgroups-3}$\Rightarrow$\eqref{item:injective-homomorphism-maximal-subgroups-1}.
\end{proof}

We are now ready to show our main result,
recovering and generalizing the case, treated in the manuscript~\cite{Kyriakoglou&Perrin:2017},
for group codes and uniformly recurrent sets.

\begin{Thm}\label{t:group-code-isomor-maximal-subgroups}
   Let $F$ be a recurrent subset of $A^*$.
  Suppose that $Z$
  is a rational bifix code of finite degree. Let $X=Z\cap F$ and suppose
  that $X$ is rational.
  The following conditions are equivalent:
  \begin{enumerate}
  \item $Z$ is $F$-charged\label{item:group-code-isomor-maximal-subgroups-1};
  \item $d_F(X)=d(Z)$, $G_F(X)\simeq G(Z)$ and
    $X$ is weakly $F$-charged\label{item:group-code-isomor-maximal-subgroups-2};
    \item $d_F(X)=d(Z)$, $|G_F(X)|=|G(Z)|$ and
    $X$ is weakly $F$-charged\label{item:group-code-isomor-maximal-subgroups-3}. 
  \end{enumerate}  
\end{Thm}

\begin{proof}
  Let $H$ be a maximal subgroup of $J(F)$.
  Denote by $H_X$ the maximal subgroup
  of $\Syn(X^\ast)$ containing $\hat\eta_{X^\ast}(H)$
  and by $H_Z$ the maximal subgroup of $\Syn(Z^\ast)$ containing $\hat\eta_{Z^\ast}(H)$.
  
    The implication
  \eqref{item:group-code-isomor-maximal-subgroups-2}$\Rightarrow$\eqref{item:group-code-isomor-maximal-subgroups-3} its trivial.
  If the equality $d_F(X)=d(Z)$ is satisfied,
  then, by Proposition~\ref{p:equal-degree-imlies-forbidden},
  the elements of~$J(F)$ are forbidden in~$\overline{Z}$ and their image
  under $\hat\eta_{Z^\ast}$ belongs to $J(Z)$,
  and so, by
  Corollary~\ref{c:injective-homomorphism-maximal-subgroups},
  we have~\eqref{item:group-code-isomor-maximal-subgroups-3}$\Rightarrow$\eqref{item:group-code-isomor-maximal-subgroups-1}.
  Also by Corollary~\ref{c:injective-homomorphism-maximal-subgroups} and Proposition~\ref{p:equal-degree-imlies-forbidden}, to prove the implication
  \eqref{item:group-code-isomor-maximal-subgroups-1}$\Rightarrow$\eqref{item:group-code-isomor-maximal-subgroups-2}
  it suffices to show that $\hat\eta_{Z^\ast}(H)=H_Z$ implies that $d_F(X)=d(Z)$.

  Suppose that $\hat\eta_{Z^\ast}(H)=H_Z$.
  Let $u\in H$ and let $e$ be the idempotent in~$H$.
  Suppose that $u$ is a factor of an element of $\overline{Z}$.
  Let $\alpha,\beta\in\F A$ be such that $\alpha u\beta\in\overline{Z}$.
  Since $u=eu$, we may as well assume that $\alpha=\alpha e$.
  Consider the minimal automaton $\mathcal M_{Z^\ast}$ of~$Z^\ast$, and let $i$ be its
  initial state, which is also the unique final state.
  Let $p=i\cdot \alpha$.
  Since $\mathcal M_{Z^\ast}$ is trim,
  there is $v\in A^*$ such that $p\cdot v=i$.
    Notice that $p=p\cdot e$, because $\alpha=\alpha e$.  
  On the other hand, $i$ is fixed by~$e$,
  by Proposition~\ref{p:idempotents-are-in-the-closure-of-Xast}.
  It follows that $p\cdot eve=i$.
  Because $\hat\eta_{Z^\ast}(H)$ is a maximal subgroup of
  the minimum ideal $J(Z)$,
  whose idempotent is $\hat\eta_{Z^\ast}(e)$, 
  there is $w\in H$
  such that $\hat\eta_{Z^\ast}(eve)=\hat\eta_{Z^\ast}(w)$.
  Notice that $i\cdot \alpha w=p\cdot w=p\cdot eve=i$,
  whence $\alpha w\in\overline{Z^+}$.
  On the other hand, as $u$ and $w$ both belong to the group $H$,
  we may consider $t\in H$ such that $u=wt$.
  Note that $\alpha u\beta=\alpha wt\beta\in \overline{Z^+A^+}$.
  By our hypothesis that $\alpha u\beta$ belongs to~$\overline {Z}$,
  and due to the rationality of $Z$, we conclude that $Z\cap Z^+ A^+\neq \emptyset$, contradicting the hypothesis
  that $Z$ is a bifix code. As the absurd originated from assuming that
  $u$ is a factor of an element of $\overline{Z}$, in view of Proposition~\ref{p:equal-degree-imlies-forbidden} this concludes
  the proof that $d_F(X)=d(Z)$ if $\hat\eta_{Z^\ast}(H)$ is a maximal subgroup
  of~$J(Z)$.  
\end{proof}

Recall that, in the setting of Theorem~\ref{t:group-code-isomor-maximal-subgroups}, if $F$ is uniformly recurrent, then $X$
is finite (cf.~Theorem~\ref{t:thin-code-intersects-F}).

\begin{Cor}\label{t:equal-degree-connected-set-group-code}
  Consider a group code $Z$ of $A^*$ and let $F$
  be a uniformly recurrent connected set with alphabet $A$.
  Take the intersection $X=Z\cap F$. We have $d_F(X)=d(Z)$ and
  $G_F(X)\simeq G(Z)$.
\end{Cor}

\begin{proof}
  It suffices to
  invoke Proposition~\ref{p:group-codes-are-connected-charged}
  and apply Theorem~\ref{t:group-code-isomor-maximal-subgroups}.
\end{proof}

\begin{Example}
  \label{eg:exampleDegree2}
  Let $A=\{a,b\}$ and consider the group code $Z=A^2$,
  whose syntactic group is the cyclic group of
  order~$2$.
  Let $F$ be the Fibonacci set. Note
  that $Z$ is $F$-charged, by Proposition~\ref{p:group-codes-are-connected-charged}.
  Therefore, the bifix code $X=F\cap Z=\{aa,ab,ba\}$
  is weakly $F$-charged by Theorem~\ref{t:group-code-isomor-maximal-subgroups}.
  The minimal automaton and the egg-box diagram of the syntactic monoid of
  $X^*$ are shown in Figure~\ref{figureAutomaton}.
  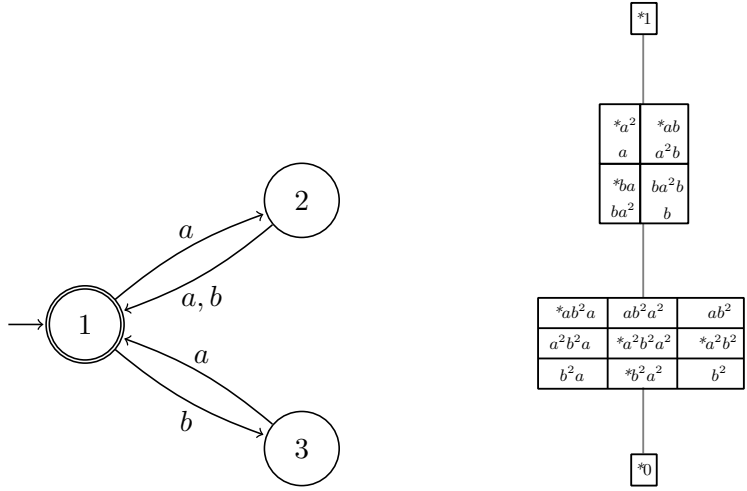
\begin{figure}[h]
    \centering
     \begin{tikzpicture}[shorten >=1pt, node distance=2.3cm and 2.7cm,
       on grid,initial text=,semithick]
       \node[state,initial,accepting]   (1)  at (180:1.9)  {$1$};
       \node[state]   (2) at (60:1.9)    {$2$};
       \node[state]   (3) at (-60:1.9)   {$3$};
       \path[->]
       (1)   edge  [bend left=10] node [above] {$a$} (2)
       (2)   edge  [bend left=10] node [below] {$a,b$} (1)
       (1)   edge  [bend right=10] node [below]  {$b$} (3)
       (3)   edge  [bend right=10] node [above]  {$a$} (1);
     \end{tikzpicture}
     \qquad\qquad\qquad     
\scalebox{.7}{
\footnotesize
\begin{tikzpicture}[line join=bevel,scale=0.8]
  \pgfsetlinewidth{1pt}
%%
  % Edge: 2 -> 1
  \pgfsetcolor{gray}
  \draw [] (78.0bp,72.946bp) .. controls (78.0bp,59.438bp) and (78.0bp,40.3bp)  .. (78.0bp,28.676bp);
  % Edge: 4 -> 3
  \draw [] (78.0bp,310.14bp) .. controls (78.0bp,298.78bp) and (78.0bp,280.25bp)  .. (78.0bp,263.2bp);
  % Edge: 3 -> 2
  \draw [] (78.0bp,182.93bp) .. controls (78.0bp,168.59bp) and (78.0bp,152.65bp)  .. (78.0bp,133.13bp);
\pgfsetcolor{black}
  \pgfsetfillcolor{white}
  % Node: 1
\begin{scope}
  \draw (78.0bp,18.0bp) node {*$0$};
  \draw (70.0bp,7.0bp) -- (70.0bp,28.0bp) -- (87.0bp,28.0bp) -- (87.0bp,7.0bp) -- cycle;
\end{scope}
  % Node: 3
\begin{scope}
  \draw (52.0bp,248.8bp) node[right] {*$a^2$};
  \draw (57.0bp,229.8bp) node[right] {$a$};
  \draw (49.0bp,223.0bp) -- (49.0bp,263.0bp) -- (76.0bp,263.0bp) -- (76.0bp,223.0bp) -- cycle;
  \draw (81.0bp,248.8bp) node[right] {*$ab$};
  \draw (81.5bp,231.8bp) node[right] {$a^2b$};
  \draw (76.0bp,223.0bp) -- (76.0bp,263.0bp) -- (108.0bp,263.0bp) -- (108.0bp,223.0bp) -- cycle;
  \draw (51.5bp,208.8bp) node[right] {*$ba$};
  \draw (52.0bp,191.8bp) node[right] {$ba^2$};
  \draw (49.0bp,183.0bp) -- (49.0bp,223.0bp) -- (76.0bp,223.0bp) -- (76.0bp,183.0bp) -- cycle;
  \draw (79.0bp,208.8bp) node[right] {$ba^2b$};
  \draw (87.0bp,189.8bp) node[right] {$b$};
  \draw (76.0bp,183.0bp) -- (76.0bp,223.0bp) -- (108.0bp,223.0bp) -- (108.0bp,183.0bp) -- cycle;
\end{scope}
  % Node: 2
\begin{scope}
  \draw (13.5bp,124.3bp) node[right] {*$ab^2a$};
  \draw (60.0bp,124.3bp) node[right] {$ab^2a^2$};
  \draw (116.5bp,124.3bp) node[right] {$ab^2$};
  \draw (11.0bp,103.3bp) node[right] {$a^2b^2a$};
  \draw (54.0bp,103.3bp) node[right] {*$a^2b^2a^2$};
  \draw (109.0bp,103.3bp) node[right] {*$a^2b^2$};
  \draw (18.0bp,82.3bp) node[right] {$b^2a$};
  \draw (59.5bp,82.3bp) node[right] {*$b^2a^2$};
  \draw (119.0bp,82.3bp) node[right] {$b^2$};
  \draw (8.0bp,72.0bp) -- (8.0bp,133.0bp) -- (145.0bp,133.0bp) -- (145.0bp,72.0bp) -- cycle;
  \draw (54.333bp,72.0bp) -- (54.333bp,133.0bp);
  \draw (100.666bp,72.0bp) -- (100.666bp,133.0bp);
  \draw (8.0bp,92.333bp) -- (145.0bp,92.333bp);
  \draw (8.0bp,112.666bp) -- (145.0bp,112.666bp);
\end{scope}
  % Node: 4
\begin{scope}
  \draw (78.0bp,321.0bp) node {*$1$};
  \draw (70.0bp,310.0bp) -- (70.0bp,331.0bp) -- (87.0bp,331.0bp) -- (87.0bp,310.0bp) -- cycle;
\end{scope}
\end{tikzpicture}
}     
\caption{Example~\ref{eg:exampleDegree2}: minimal automaton of $X^*$ and egg-box diagram of its transition monoid.}
\label{figureAutomaton}
\end{figure}
  Since $d_F(X)=2$, the $F$-minimum $\mathcal J$-class of $\Syn(X^*)$
  has rank~$2$, and so $J_F(X)$ is the $\mathcal J$-class of $\eta_{X^*}(a)$.
  This can be seen also
  using Corollary~\ref{t:equal-degree-connected-set-group-code}.
  Indeed, the group $G_F(X)$ is the cyclic group of order~$2$, as expected by Corollary~\ref{t:equal-degree-connected-set-group-code},
  and so $J_F(X)$ must be the unique $\mathcal J$-class of $\Syn(X^*)$ with a maximal subgroup
  of order $2$.
\end{Example}

The special case of Corollary~\ref{t:equal-degree-connected-set-group-code}
for uniformly recurrent tree sets
is established in~\cite{Kyriakoglou&Perrin:2017} as a consequence of the specialization of Theorem~\ref{t:group-code-isomor-maximal-subgroups} for the setting of uniformly recurrent sets and group codes.
The special case of Corollary~\ref{t:equal-degree-connected-set-group-code} in which $Z$ is a group code  and $F$ is a Sturmian set
was first proved in~\cite{Berstel&Felice&Perrin&Reutenauer&Rindone:2012}.

In the next example we see that both conclusions
of Corollary~\ref{t:equal-degree-connected-set-group-code} fail
if we only require $F$ to be uniformly recurrent, even if $Z$
is still a group code.

\begin{Example}
  Consider the group code $Z$ and the uniformly recurrent set~$F$
  of Example~\ref{eg:non-forbidden-example},  and let $X=Z\cap F$.
  We saw in Example~\ref{eg:non-forbidden-example} that $d_F(X)<d(Z)=4$.
  Therefore,
  by Corollary~\ref{c:upper-bound-size-FGroups}, we have $G(Z)\leq S_4$  and
  $G_F(X)\leq S_3$. Actually,
  a direct computation shows that $G(Z)\simeq S_4$.
\end{Example}

The following is an example of application of Theorem~\ref{t:group-code-isomor-maximal-subgroups} to the setting of recurrent sets which are not
uniformly recurrent and of maximal bifix codes which are not group codes.
It is based on \cite[Example 4.2.15]{Berstel&Felice&Perrin&Reutenauer&Rindone:2012}.

\begin{Example}\label{eg:naive-sofic-example}
  Let $A=\{a,b\}$ and $Z=\{aa,ab,ba\}\cup b^2(a^+b)^\ast b$.
  The language $Z$ is a maximal bifix code
  of degree $3$.
   It is not a group code.
   This can be seen by applying Theorem~\ref{t:equal-degree-tree-set-group-code},
   because if $F$ is the Fibonacci set, then $Z\cap F=\{aa,ab,ba\}$
   has $F$-degree $2$.
   \begin{figure}[h]
     \centering
     \begin{tikzpicture}[shorten >=1pt, node distance=2.3cm and 2.7cm,
       on grid,initial text=,semithick]
       \node[state,initial,accepting]   (1)   {$1$};
       \node[state]   (2) [right=of 1]    {$2$};
       \node[state]   (3) [below=of 1]   {$3$};
       \node[state]   (4) [right=of 3]   {$4$};
       \node[state]   (5) [below=of 4]   {$5$};
       \path[->]
       (1)   edge  [bend right=10] node [below] {$a$} (2)
       (2)   edge  [bend right=10] node [above] {$a,b$} (1)
       (1)   edge  [bend right=10] node [left]  {$b$} (3)
       (3)   edge  [bend right=10] node [below] {$b$} (4)
       (3)   edge  [bend right=10] node [right] {$a$} (1)
       (4)   edge                  node [below] {$b$} (1)
       (4)   edge  [bend right=10] node [left]  {$a$} (5)
       (5)   edge  [bend right=10] node [right] {$b$} (4)
       (5)   edge  [loop left]     node         {$a$} ();
     \end{tikzpicture}
     \qquad\qquad\qquad     
\footnotesize
\begin{tikzpicture}[>=latex,line join=bevel,scale=0.6]
  \definecolor{fillcol}{rgb}{1.0,1.0,1.0};
  \pgfsetfillcolor{fillcol}
  \definecolor{strokecol}{rgb}{0.0,0.0,0.0};
  \pgfsetstrokecolor{strokecol}
  \pgfsetlinewidth{1pt}
%%
  % Edge: 2 -> 1
  \pgfsetcolor{gray}
  \draw [] (61.5bp,283.38bp) .. controls (61.5bp,272.9bp) and
  (61.5bp,255.67bp)  .. (61.5bp,236.23bp);
  \pgfsetlinewidth{0.4pt}
  \pgfsetcolor{black}
  % Node: 1
\begin{scope}
  \draw (21.5bp,221.8bp) node[right] {*$a^2$};
  \draw (18.5bp,202.8bp) node[right] {$ab^2a$};
  \draw (31.5bp,183.8bp) node[right] {$a$};
  \draw (11.5bp,164.8bp) node[right] {$ab^2a^2$};
  \draw (11.5bp,145.8bp) node[right] {$a^2b^2a$};
  \draw (18.5bp,126.8bp) node[right] {$ab^3a$};
  \draw (78bp,221.8bp) node[right] {*$ab$};
  \draw (68bp,202.8bp) node[right] {$a^2b^3$};
  \draw (75bp,183.8bp) node[right] {$a^2b$};
  \draw (75bp,164.8bp) node[right] {$ab^2$};
  \draw (75bp,145.8bp) node[right] {$ab^3$};
  \draw (68bp,126.8bp) node[right] {$a^2b^2$};
  \draw (24.5bp,105.8bp) node[right] {*$ba$};
  \draw (14.5bp,86.8bp) node[right] {$b^3a^2$};
  \draw (21.5bp,67.8bp) node[right] {$ba^2$};
  \draw (21.5bp,48.8bp) node[right] {$b^3a$};
  \draw (21.5bp,29.8bp) node[right] {$b^2a$};
  \draw (14.5bp,10.8bp) node[right] {$b^2a^2$};
  \draw (71.5bp,105.8bp) node[right] {$ba^2b$};
  \draw (74.5bp,86.8bp) node[right] {*$b^3$};
  \draw (84.5bp,67.8bp) node[right] {$b$};
  \draw (64.5bp,48.8bp) node[right] {$ba^2b^2$};
  \draw (64.5bp,29.8bp) node[right] {$b^2a^2b$};
  \draw (78bp,10.8bp) node[right] {$b^2$};
  \draw (8bp,0bp) -- (8bp,236bp) -- (118bp,236bp) --
  (118bp,0bp) -- cycle;
  \draw (63bp,0bp) -- (63bp,236bp);
  \draw (8bp,118bp) -- (118bp,118bp);
\end{scope}
  % Node: 2
\begin{scope}
  \draw (61.5bp,294bp) node {*1};
  \draw (53bp,283bp) -- (73bp,283bp) --
  (73bp,305bp) -- (53bp,305bp) -- cycle;
\end{scope}
\end{tikzpicture}
    \caption{Example~\ref{eg:naive-sofic-example}: minimal automaton of $Z^\ast$ and
      egg-box diagram of its transition monoid.}
    \label{fig:first-automaton}
  \end{figure}
   
  The monoid $\Syn(Z^\ast)$ is the disjoint union
  of its unit with its minimum ideal: see Figure~\ref{fig:first-automaton}.
  Let $E$ be the recurrent set
  formed by the words labeling finite paths
  of the labeled graph $\mathcal G$ in Figure~\ref{fig:even-shift}
  (this is the set of factors of the so called \emph{even subshift}, see~\cite{Lind&Marcus:1996}).
     \begin{figure}[h]
     \centering
     \begin{tikzpicture}[shorten >=1pt, node distance=3cm, on grid,initial text=,semithick]
       % \tikzstyle{accepting}=[accepting by arrow]
  \node[state]   (1)                {$1$};
  \node[state]   (2) [right=of 1]   {$2$};
  \path[->]  (1)   edge  [bend right=30] node [below] {$b$} (2)
             (2)   edge  [bend right=30]  node [above] {$b$} (1)
             (1)      edge [loop left]   node        {$a$} ();
\end{tikzpicture}
     \caption{Presentation of the even subshift.}
     \label{fig:even-shift}
   \end{figure}
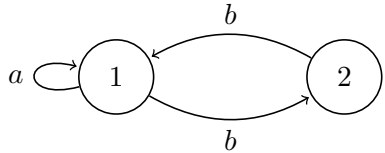
  The set $X=Z\cap E$ is given by
  $X=\{aa,ab,ba,bbb\}\cup b^2a^+b^2$.
  Since $a\in E$ and $E$ is recurrent,
  there is $u\in J(E)$ such that $u\in a\cdot \F A\cdot a$.
  Let $v=u^\omega ab^2a^2 u^\omega$. Because the action of $a$ in $\mathcal G$
  consists in just fixing $1$,
  we have $1=1\cdot ab^2a^2=1\cdot u$ in~$\mathcal G$. 
  Hence, $u^\omega$ and $v$ belong to $J(E)$.
Denote by $H$ the maximal subgroup of $J(E)$ containing~$u^\omega$.
Notice that $v\in H$.
Similarly, for $w=u^\omega a u^\omega$ we have $w\in H$. 
The maximal subgroup $K$ of
$\Syn(Z^\ast)$ containing $\hat\eta_{Z^\ast}(a)$ is isomorphic to
$S_3$
and it is generated by $\{\hat\eta_{Z^\ast}(a),\hat\eta_{Z^\ast}(ab^2a^2)\}$.
Clearly, one has $\hat\eta_{Z^\ast}(v)=\hat\eta_{Z^\ast}(ab^2a^2)$
and $\hat\eta_{Z^\ast}(w)=\hat\eta_{Z^\ast}(a)$.
We conclude that $\hat\eta_{Z^\ast}(H)=K$,
thus $Z$ is $E$-charged.
Applying Theorem~\ref{t:group-code-isomor-maximal-subgroups},
we deduce that $d_E(X)=3$ and that $G_E(X)$ is
isomorphic to~$S_3$; see Figure~\ref{fig:second-automaton}
on page~\pageref{fig:second-automaton}.
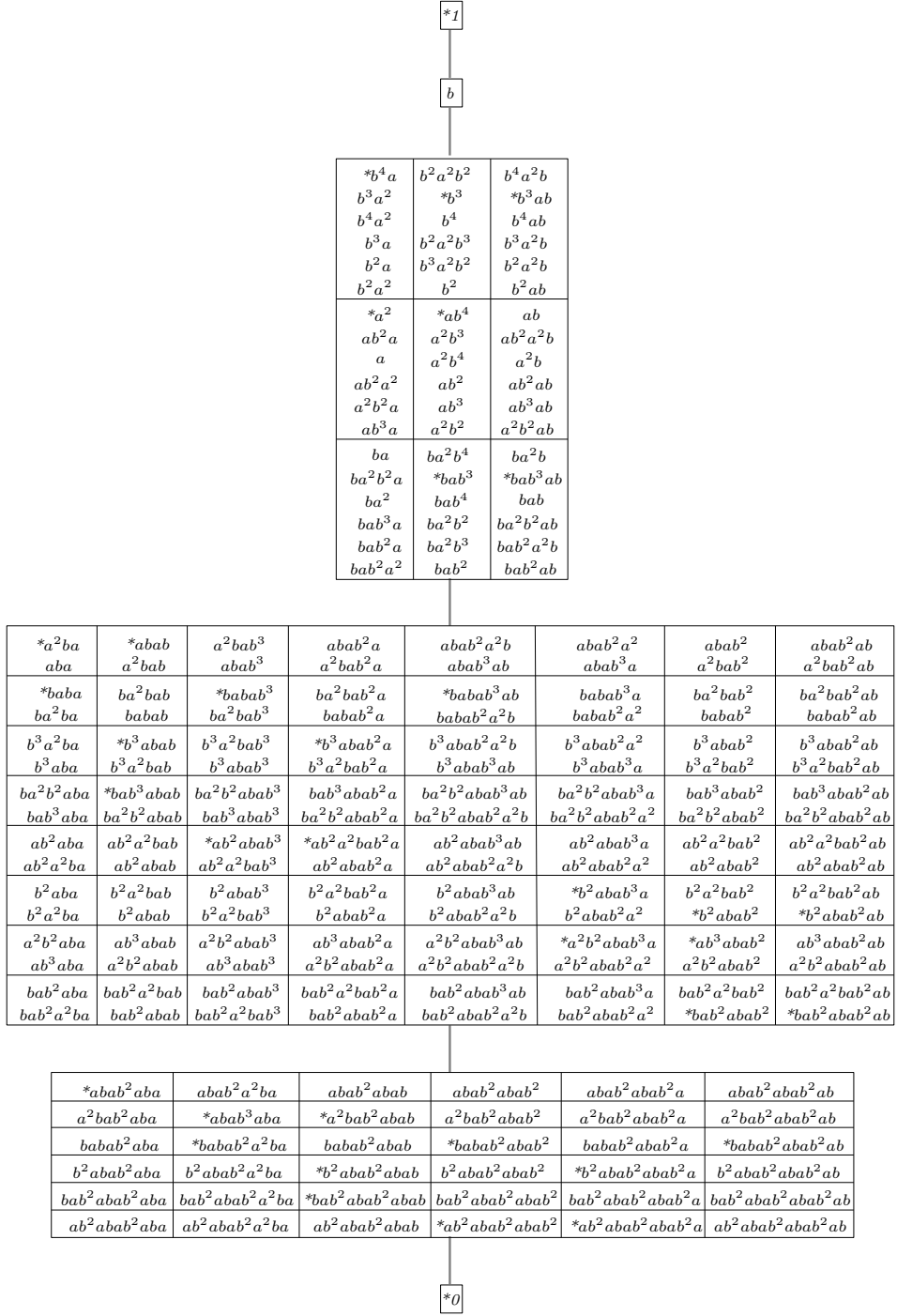
\begin{figure}[htp]
    \centering
\scalebox{.9}{
\tiny
\begin{tikzpicture}[>=latex,line join=bevel,scale=0.48]
  \pgfsetlinewidth{1pt}
\pgfsetcolor{black}
  % Edge: 2 -> 1
  \pgfsetcolor{gray}
  \draw [] (416.5bp,75.995bp) .. controls (416.5bp,58.983bp) and (416.5bp,42.759bp)  .. (416.5bp,31.82bp);
  % Edge: 4 -> 3
  \draw [] (416.5bp,685.98bp) .. controls (416.5bp,671.4bp) and (416.5bp,656.71bp)  .. (416.5bp,642.19bp);
  % Edge: 5 -> 4
  \draw [] (416.5bp,1122.3bp) .. controls (416.5bp,1112bp) and (416.5bp,1096.4bp)  .. (416.5bp,1078.1bp);
  % Edge: 3 -> 2
  \draw [] (416.5bp,271.67bp) .. controls (416.5bp,256.45bp) and (416.5bp,241.69bp)  .. (416.5bp,228.03bp);
  % Edge: 6 -> 5
  \draw [] (416.5bp,1194.5bp) .. controls (416.5bp,1181.7bp) and (416.5bp,1162.8bp)  .. (416.5bp,1150bp);
\pgfsetcolor{black}
  \pgfsetlinewidth{0.4pt}
  % Node: 1
\begin{scope}
  \draw (416.5bp,18bp) node {*0};
  \draw (408bp,5bp) -- (428bp,5bp) -- (428bp,31bp)
  -- (408bp,31bp) -- cycle;
\end{scope}
  % Node: 3
\begin{scope}
  \draw (28bp,625.8bp) node[right] {*$a^2ba$};
  \draw (38bp,604.8bp) node[right] {$aba$};
  \draw (111.5bp,625.8bp) node[right] {*$abab$};
  \draw (108.5bp,606.8bp) node[right] {$a^2bab$};
  \draw (192.5bp,625.8bp) node[right] {$a^2bab^3$};
  \draw (199.5bp,606.8bp) node[right] {$abab^3$};
  \draw (297bp,625.8bp) node[right] {$abab^2a$};
  \draw (290.5bp,606.8bp) node[right] {$a^2bab^2a$};
  \draw (401.5bp,625.8bp) node[right] {$abab^2a^2b$};
  \draw (408bp,606.8bp) node[right] {$abab^3ab$};
  \draw (526.5bp,625.8bp) node[right] {$abab^2a^2$};
  \draw (533bp,606.8bp) node[right] {$abab^3a$};
  \draw (644.5bp,625.8bp) node[right] {$abab^2$};
  \draw (637.5bp,606.8bp) node[right] {$a^2bab^2$};
  \draw (742bp,625.8bp) node[right] {$abab^2ab$};
  \draw (735.5bp,606.8bp) node[right] {$a^2bab^2ab$};
  \draw (31bp,579.8bp) node[right] {*$baba$};
  \draw (28bp,560.8bp) node[right] {$ba^2ba$};
  \draw (105bp,579.8bp) node[right] {$ba^2bab$};
  \draw (111.5bp,558.8bp) node[right] {$babab$};
  \draw (192.5bp,579.8bp) node[right] {*$babab^3$};
  \draw (189bp,560.8bp) node[right] {$ba^2bab^3$};
  \draw (287bp,579.8bp) node[right] {$ba^2bab^2a$};
  \draw (293.5bp,560.8bp) node[right] {$babab^2a$};
  \draw (401bp,579.8bp) node[right] {*$babab^3ab$};
  \draw (398bp,558.8bp) node[right] {$babab^2a^2b$};
  \draw (529.5bp,579.8bp) node[right] {$babab^3a$};
  \draw (523bp,560.8bp) node[right] {$babab^2a^2$};
  \draw (634bp,579.8bp) node[right] {$ba^2bab^2$};
  \draw (641bp,560.8bp) node[right] {$babab^2$};
  \draw (732bp,579.8bp) node[right] {$ba^2bab^2ab$};
  \draw (738.5bp,560.8bp) node[right] {$babab^2ab$};
  \draw (21bp,533.8bp) node[right] {$b^3a^2ba$};
  \draw (28bp,512.8bp) node[right] {$b^3aba$};
  \draw (101.5bp,533.8bp) node[right] {*$b^3abab$};
  \draw (98bp,512.8bp) node[right] {$b^3a^2bab$};
  \draw (182bp,533.8bp) node[right] {$b^3a^2bab^3$};
  \draw (189bp,512.8bp) node[right] {$b^3abab^3$};
  \draw (283.5bp,533.8bp) node[right] {*$b^3abab^2a$};
  \draw (280bp,512.8bp) node[right] {$b^3a^2bab^2a$};
  \draw (391bp,533.8bp) node[right] {$b^3abab^2a^2b$};
  \draw (398bp,512.8bp) node[right] {$b^3abab^3ab$};
  \draw (516bp,533.8bp) node[right] {$b^3abab^2a^2$};
  \draw (523bp,512.8bp) node[right] {$b^3abab^3a$};
  \draw (634bp,533.8bp) node[right] {$b^3abab^2$};
  \draw (627bp,512.8bp) node[right] {$b^3a^2bab^2$};
  \draw (732bp,533.8bp) node[right] {$b^3abab^2ab$};
  \draw (725bp,512.8bp) node[right] {$b^3a^2bab^2ab$};
  \draw (14.5bp,487.8bp) node[right] {$ba^2b^2aba$};
  \draw (21bp,466.8bp) node[right] {$bab^3aba$};
  \draw (89.5bp,487.8bp) node[right] {*$bab^3abab$};
  \draw (91.5bp,466.8bp) node[right] {$ba^2b^2abab$};
  \draw (175.5bp,487.8bp) node[right] {$ba^2b^2abab^3$};
  \draw (182.5bp,466.8bp) node[right] {$bab^3abab^3$};
  \draw (280.5bp,487.8bp) node[right] {$bab^3abab^2a$};
  \draw (273.5bp,466.8bp) node[right] {$ba^2b^2abab^2a$};
  \draw (384.5bp,487.8bp) node[right] {$ba^2b^2abab^3ab$};
  \draw (377.5bp,466.8bp) node[right] {$ba^2b^2abab^2a^2b$};
  \draw (509.5bp,487.8bp) node[right] {$ba^2b^2abab^3a$};
  \draw (502.5bp,466.8bp) node[right] {$ba^2b^2abab^2a^2$};
  \draw (627.5bp,487.8bp) node[right] {$bab^3abab^2$};
  \draw (620.5bp,466.8bp) node[right] {$ba^2b^2abab^2$};
  \draw (725.5bp,487.8bp) node[right] {$bab^3abab^2ab$};
  \draw (718.5bp,466.8bp) node[right] {$ba^2b^2abab^2ab$};
  \draw (24.5bp,441.8bp) node[right] {$ab^2aba$};
  \draw (18bp,420.8bp) node[right] {$ab^2a^2ba$};
  \draw (95bp,441.8bp) node[right] {$ab^2a^2bab$};
  \draw (101.5bp,420.8bp) node[right] {$ab^2abab$};
  \draw (182.5bp,441.8bp) node[right] {*$ab^2abab^3$};
  \draw (179bp,420.8bp) node[right] {$ab^2a^2bab^3$};
  \draw (273.5bp,441.8bp) node[right] {*$ab^2a^2bab^2a$};
  \draw (284bp,420.8bp) node[right] {$ab^2abab^2a$};
  \draw (395bp,441.8bp) node[right] {$ab^2abab^3ab$};
  \draw (388bp,420.8bp) node[right] {$ab^2abab^2a^2b$};
  \draw (520bp,441.8bp) node[right] {$ab^2abab^3a$};
  \draw (513bp,420.8bp) node[right] {$ab^2abab^2a^2$};
  \draw (624bp,441.8bp) node[right] {$ab^2a^2bab^2$};
  \draw (631bp,420.8bp) node[right] {$ab^2abab^2$};
  \draw (722bp,441.8bp) node[right] {$ab^2a^2bab^2ab$};
  \draw (729bp,420.8bp) node[right] {$ab^2abab^2ab$};
  \draw (28bp,395.8bp) node[right] {$b^2aba$};
  \draw (21bp,374.8bp) node[right] {$b^2a^2ba$};
  \draw (98bp,395.8bp) node[right] {$b^2a^2bab$};
  \draw (105bp,374.8bp) node[right] {$b^2abab$};
  \draw (189bp,395.8bp) node[right] {$b^2abab^3$};
  \draw (182bp,374.8bp) node[right] {$b^2a^2bab^3$};
  \draw (280bp,395.8bp) node[right] {$b^2a^2bab^2a$};
  \draw (287bp,374.8bp) node[right] {$b^2abab^2a$};
  \draw (398bp,395.8bp) node[right] {$b^2abab^3ab$};
  \draw (391bp,374.8bp) node[right] {$b^2abab^2a^2b$};
  \draw (519.5bp,395.8bp) node[right] {*$b^2abab^3a$};
  \draw (516bp,374.8bp) node[right] {$b^2abab^2a^2$};
  \draw (627bp,395.8bp) node[right] {$b^2a^2bab^2$};
  \draw (630.5bp,374.8bp) node[right] {*$b^2abab^2$};
  \draw (725bp,395.8bp) node[right] {$b^2a^2bab^2ab$};
  \draw (728.5bp,374.8bp) node[right] {*$b^2abab^2ab$};
  \draw (18bp,349.8bp) node[right] {$a^2b^2aba$};
  \draw (24.5bp,328.8bp) node[right] {$ab^3aba$};
  \draw (101.5bp,349.8bp) node[right] {$ab^3abab$};
  \draw (95bp,328.8bp) node[right] {$a^2b^2abab$};
  \draw (179bp,349.8bp) node[right] {$a^2b^2abab^3$};
  \draw (186bp,328.8bp) node[right] {$ab^3abab^3$};
  \draw (284bp,349.8bp) node[right] {$ab^3abab^2a$};
  \draw (277bp,328.8bp) node[right] {$a^2b^2abab^2a$};
  \draw (388bp,349.8bp) node[right] {$a^2b^2abab^3ab$};
  \draw (381bp,328.8bp) node[right] {$a^2b^2abab^2a^2b$};
  \draw (509.5bp,349.8bp) node[right] {*$a^2b^2abab^3a$};
  \draw (506bp,328.8bp) node[right] {$a^2b^2abab^2a^2$};
  \draw (627.5bp,349.8bp) node[right] {*$ab^3abab^2$};
  \draw (624bp,328.8bp) node[right] {$a^2b^2abab^2$};
  \draw (729bp,349.8bp) node[right] {$ab^3abab^2ab$};
  \draw (722bp,328.8bp) node[right] {$a^2b^2abab^2ab$};
  \draw (21bp,303.8bp) node[right] {$bab^2aba$};
  \draw (14.5bp,282.8bp) node[right] {$bab^2a^2ba$};
  \draw (91.5bp,303.8bp) node[right] {$bab^2a^2bab$};
  \draw (98bp,282.8bp) node[right] {$bab^2abab$};
  \draw (182.5bp,303.8bp) node[right] {$bab^2abab^3$};
  \draw (175.5bp,282.8bp) node[right] {$bab^2a^2bab^3$};
  \draw (273.5bp,303.8bp) node[right] {$bab^2a^2bab^2a$};
  \draw (280.5bp,282.8bp) node[right] {$bab^2abab^2a$};
  \draw (391.5bp,303.8bp) node[right] {$bab^2abab^3ab$};
  \draw (384.5bp,282.8bp) node[right] {$bab^2abab^2a^2b$};
  \draw (516.5bp,303.8bp) node[right] {$bab^2abab^3a$};
  \draw (509.5bp,282.8bp) node[right] {$bab^2abab^2a^2$};
  \draw (620.5bp,303.8bp) node[right] {$bab^2a^2bab^2$};
  \draw (624bp,282.8bp) node[right] {*$bab^2abab^2$};
  \draw (718.5bp,303.8bp) node[right] {$bab^2a^2bab^2ab$};
  \draw (718bp,282.8bp) node[right] {*$bab^2abab^2ab$};
  \draw (9bp,272bp) -- (9bp,642bp) -- (825bp,642bp) --
  (825bp,272bp) -- cycle;
  \draw (9bp,318.25bp) -- (825bp,318.25bp);
  \draw (9bp,364.5bp) -- (825bp,364.5bp);
  \draw (9bp,410.75bp) -- (825bp,410.75bp);
  \draw (9bp,457bp) -- (825bp,457bp);
  \draw (9bp,503.25bp) -- (825bp,503.25bp);
  \draw (9bp,549.5bp) -- (825bp,549.5bp);
  \draw (9bp,595.75bp) -- (825bp,595.75bp);
  \draw (92bp,272bp) -- (92bp,642bp);
  \draw (175bp,272bp) -- (175bp,642bp);
  \draw (268bp,272bp) -- (268bp,642bp);
  \draw (375bp,272bp) -- (375bp,642bp);
  \draw (497bp,272bp) -- (495bp,642bp);
  \draw (614bp,272bp) -- (614bp,642bp);
  \draw (716bp,272bp) -- (716bp,642bp);
\end{scope}
  % Node: 2
\begin{scope}
  \draw (70bp,211.8bp) node[right] {*$abab^2aba$};
  \draw (177.5bp,211.8bp) node[right] {$abab^2a^2ba$};
  \draw (298bp,211.8bp) node[right] {$abab^2abab$};
  \draw (412.5bp,211.8bp) node[right] {$abab^2abab^2$};
  \draw (536bp,211.8bp) node[right] {$abab^2abab^2a$};
  \draw (666bp,211.8bp) node[right] {$abab^2abab^2ab$};
  \draw (67bp,186.8bp) node[right] {$a^2bab^2aba$};
  \draw (180.5bp,186.8bp) node[right] {*$abab^3aba$};
  \draw (288bp,186.8bp) node[right] {*$a^2bab^2abab$};
  \draw (405.5bp,186.8bp) node[right] {$a^2bab^2abab^2$};
  \draw (529.5bp,186.8bp) node[right] {$a^2bab^2abab^2a$};
  \draw (659.5bp,186.8bp) node[right] {$a^2bab^2abab^2ab$};
  \draw (70bp,161.8bp) node[right] {$babab^2aba$};
  \draw (170.5bp,161.8bp) node[right] {*$babab^2a^2ba$};
  \draw (294.5bp,161.8bp) node[right] {$babab^2abab$};
  \draw (405.5bp,161.8bp) node[right] {*$babab^2abab^2$};
  \draw (532.5bp,161.8bp) node[right] {$babab^2abab^2a$};
  \draw (659bp,161.8bp) node[right] {*$babab^2abab^2ab$};
  \draw (63.5bp,136.8bp) node[right] {$b^2abab^2aba$};
  \draw (167bp,136.8bp) node[right] {$b^2abab^2a^2ba$};
  \draw (284.5bp,136.8bp) node[right] {*$b^2abab^2abab$};
  \draw (402bp,136.8bp) node[right] {$b^2abab^2abab^2$};
  \draw (522.5bp,136.8bp) node[right] {*$b^2abab^2abab^2a$};
  \draw (656bp,136.8bp) node[right] {$b^2abab^2abab^2ab$};
  \draw (52.5bp,111.8bp) node[right] {$bab^2abab^2aba$};
  \draw (161.5bp,111.8bp) node[right] {$bab^2abab^2a^2ba$};
  \draw (274.5bp,111.8bp) node[right] {*$bab^2abab^2abab$};
  \draw (397.5bp,111.8bp) node[right] {$bab^2abab^2abab^2$};
  \draw (519.5bp,111.8bp) node[right] {$bab^2abab^2abab^2a$};
  \draw (649.5bp,111.8bp) node[right] {$bab^2abab^2abab^2ab$};
  \draw (60bp,86.8bp) node[right] {$ab^2abab^2aba$};
  \draw (164bp,86.8bp) node[right] {$ab^2abab^2a^2ba$};
  \draw (284.5bp,86.8bp) node[right] {$ab^2abab^2abab$};
  \draw (395.5bp,86.8bp) node[right] {*$ab^2abab^2abab^2$};
  \draw (519.5bp,86.8bp) node[right] {*$ab^2abab^2abab^2a$};
  \draw (653bp,86.8bp) node[right] {$ab^2abab^2abab^2ab$};
  \draw (50bp,75bp) -- (50bp,228bp) -- (788bp,228bp) --
  (788bp,75bp) -- cycle;
  \draw (162bp,75bp) -- (162bp,228bp);
  \draw (278bp,75bp) -- (278bp,228bp);
  \draw (399bp,75bp) -- (399bp,228bp);
  \draw (519bp,75bp) -- (519bp,228bp);
  \draw (651bp,75bp) -- (651bp,228bp);
  \draw (50bp,100.5bp) -- (788bp,100.5bp);
  \draw (50bp,126bp) -- (788bp,126bp);
  \draw (50bp,151.5bp) -- (788bp,151.5bp);
  \draw (50bp,177bp) -- (788bp,177bp);
  \draw (50bp,201.5bp) -- (788bp,201.5bp);
\end{scope}
  % Node: 5
\begin{scope}
  \draw (416.5bp,1136bp) node {b};
  \draw (408bp,1123bp) -- (428bp,1123bp) -- (428bp,1149bp) --
  (408bp,1149bp) -- cycle;
\end{scope}
  % Node: 4
\begin{scope}
  \draw (328.5bp,1060.8bp) node[right] {*$b^4a$};
  \draw (325bp,1039.8bp) node[right] {$b^3a^2$};
  \draw (325bp,1018.8bp) node[right] {$b^4a^2$};
  \draw (332bp,997.8bp) node[right] {$b^3a$};
  \draw (332bp,976.8bp) node[right] {$b^2a$};
  \draw (325bp,955.8bp) node[right] {$b^2a^2$};
  \draw (382.5bp,1060.8bp) node[right] {$b^2a^2b^2$};
  \draw (399bp,1039.8bp) node[right] {*$b^3$};
  \draw (402.5bp,1018.8bp) node[right] {$b^4$};
  \draw (382.5bp,997.8bp) node[right] {$b^2a^2b^3$};
  \draw (382.5bp,976.8bp) node[right] {$b^3a^2b^2$};
  \draw (402.5bp,955.8bp) node[right] {$b^2$};
  \draw (460bp,1060.8bp) node[right] {$b^4a^2b$};
  \draw (463.5bp,1039.8bp) node[right] {*$b^3ab$};
  \draw (467bp,1018.8bp) node[right] {$b^4ab$};
  \draw (460bp,997.8bp) node[right] {$b^3a^2b$};
  \draw (460bp,976.8bp) node[right] {$b^2a^2b$};
  \draw (467bp,955.8bp) node[right] {$b^2ab$};
  \draw (332bp,930.8bp) node[right] {*$a^2$};
  \draw (329bp,909.8bp) node[right] {$ab^2a$};
  \draw (342bp,888.8bp) node[right] {$a$};
  \draw (322bp,867.8bp) node[right] {$ab^2a^2$};
  \draw (322bp,846.8bp) node[right] {$a^2b^2a$};
  \draw (329bp,825.8bp) node[right] {$ab^3a$};
  \draw (396bp,930.8bp) node[right] {*$ab^4$};
  \draw (392.5bp,909.8bp) node[right] {$a^2b^3$};
  \draw (392.5bp,888.8bp) node[right] {$a^2b^4$};
  \draw (399.5bp,867.8bp) node[right] {$ab^2$};
  \draw (399.5bp,846.8bp) node[right] {$ab^3$};
  \draw (392.5bp,825.8bp) node[right] {$a^2b^2$};
  \draw (477bp,930.8bp) node[right] {$ab$};
  \draw (457bp,909.8bp) node[right] {$ab^2a^2b$};
  \draw (470.5bp,888.8bp) node[right] {$a^2b$};
  \draw (464bp,867.8bp) node[right] {$ab^2ab$};
  \draw (464bp,846.8bp) node[right] {$ab^3ab$};
  \draw (457bp,825.8bp) node[right] {$a^2b^2ab$};
  \draw (338.5bp,800.8bp) node[right] {$ba$};
  \draw (318.5bp,779.8bp) node[right] {$ba^2b^2a$};
  \draw (332bp,758.8bp) node[right] {$ba^2$};
  \draw (325.5bp,737.8bp) node[right] {$bab^3a$};
  \draw (325.5bp,716.8bp) node[right] {$bab^2a$};
  \draw (318.5bp,695.8bp) node[right] {$bab^2a^2$};
  \draw (389bp,800.8bp) node[right] {$ba^2b^4$};
  \draw (392.5bp,779.8bp) node[right] {*$bab^3$};
  \draw (396bp,758.8bp) node[right] {$bab^4$};
  \draw (389bp,737.8bp) node[right] {$ba^2b^2$};
  \draw (389bp,716.8bp) node[right] {$ba^2b^3$};
  \draw (396bp,695.8bp) node[right] {$bab^2$};
  \draw (467bp,800.8bp) node[right] {$ba^2b$};
  \draw (457bp,779.8bp) node[right] {*$bab^3ab$};
  \draw (473.5bp,758.8bp) node[right] {$bab$};
  \draw (453.5bp,737.8bp) node[right] {$ba^2b^2ab$};
  \draw (453.5bp,716.8bp) node[right] {$bab^2a^2b$};
  \draw (460.5bp,695.8bp) node[right] {$bab^2ab$};
  \draw (312bp,685bp) -- (312bp,1075bp) -- (525bp,1075bp) --
  (525bp,685bp) -- cycle;
  \draw (383bp,685bp) -- (383bp,1075bp);
  \draw (454bp,685bp) -- (454bp,1075bp);
  \draw (312bp,815bp) -- (525bp,815bp);
  \draw (312bp,945bp) -- (525bp,945bp);
\end{scope}
  % Node: 6
\begin{scope}
  \draw (416.5bp,1208bp) node {*1};
  \draw (408bp,1195bp) -- (428bp,1195bp) -- (428bp,1221bp) --
  (408bp,1221bp) -- cycle;
\end{scope}
\end{tikzpicture}
}
\caption{Example~\ref{eg:naive-sofic-example}: egg-box diagram of $\Syn(X^\ast)$, a monoid with $221$ elements.}
      \label{fig:second-automaton}
  \end{figure}
  The computations of Figures~\ref{fig:first-automaton} and~\ref{fig:second-automaton} were carried out using \pv{GAP}~\cite{GAP4:2013,Delgado&Linton&Morais:automata,Delgado&Morais:sgpviz}.
\end{Example}

In the case where $F$ is uniformly recurrent,
we have the following variation of
Theorem~\ref{t:group-code-isomor-maximal-subgroups}.
A finite semigroup is said to be \emph{nil-simple}
if all its idempotents lie in the minimum ideal.

\begin{Thm}\label{t:uniformly-recurr-version-group-code-isomor-maximal-subgroups}
  Let $F$ be a uniformly recurrent subset of $A^*$.
  Suppose that $Z$ is a rational bifix code of finite degree such that
  $\eta_{Z^\ast}(A^+)$ is nil-simple, and let $X=Z\cap F$.
  The following conditions are equivalent:
  \begin{enumerate}
  \item $Z$ is $F$-charged\label{item:uni-rec-vers-1};
  \item $G_F(X)\simeq G(Z)$ and
    $X$ is weakly $F$-charged\label{item:uni-rec-vers-2};
  \item $|G_F(X)|=|G(Z)|$ and $X$ is weakly $F$-charged\label{item:uni-rec-vers-3}.
  \end{enumerate}  
\end{Thm}

As examples of rational bifix codes $Z$
such that $\eta_{Z^\ast}(A^+)$ is nil-simple,
we have the finite maximal bifix codes~(cf.~\cite[Theorem 11.5.2]{Berstel&Perrin&Reutenauer:2010}) and the group codes. There
are rational maximal bifix codes such that $\eta_{Z^\ast}(A^+)$
is not nil-simple~(cf.~\cite[Example 11.5.3]{Berstel&Perrin&Reutenauer:2010}).

\begin{proof}[Proof of Theorem~\ref{t:uniformly-recurr-version-group-code-isomor-maximal-subgroups}]  
  Implication~\eqref{item:uni-rec-vers-1}$\Rightarrow$\eqref{item:uni-rec-vers-2} holds by Theorem~\ref{t:group-code-isomor-maximal-subgroups},
  and implication~\eqref{item:uni-rec-vers-2}$\Rightarrow$\eqref{item:uni-rec-vers-3} is trivial. Suppose that \eqref{item:uni-rec-vers-3} holds.
  Take a maximal subgroup $H$ of $J(F)$.
  Let $H_X=\hat\eta_{X^\ast}(H)$
  and denote by $H_Z$ the maximal subgroup of $\Syn(Z^\ast)$
  containing $\hat\eta_{Z^\ast}(H)$.
  Note that $H_Z\subseteq J(Z)$, because
  the semigroup $\eta_{Z^\ast}(A^+)=\hat\eta_{Z^\ast}(\widehat {A^+})$ is nil-simple.
  Since $X$ is finite (cf.~Theorem~\ref{t:thin-code-intersects-F}),
  the elements of $H$ are forbidden in $\overline{X}=X$.
  As seen in the proof of Corollary~\ref{c:isomorphism-between-images},
  it follows from Proposition~\ref{p:fundamental-syntactic-inequalities}\eqref{item:fundamental-syntactic-inequalities-2} that there is a homomorphism $\beta\colon H_Z\to H_X$
  such that the diagram
   \begin{equation*}
    \xymatrix{
      H\ar[r]^{\hat\eta_{Z^\ast}}\ar[d]_{\hat\eta_{X^\ast}}&H_Z\ar[dl]^\beta\\
      H_X&
    }
  \end{equation*}
  commutes.
  Notice that $\beta$ is onto, as
  $H_X=\hat\eta_{X^\ast}(H)$.
  Therefore, $\beta$ is an isomorphism, because $|H_X|=|H_Z|$.
  It follows that $H_Z=\hat\eta_{Z^\ast}(H)$.
  This establishes \eqref{item:uni-rec-vers-3}$\Rightarrow$\eqref{item:uni-rec-vers-1}, concluding the proof of the theorem.
\end{proof}

\section{$F$-groups as permutation groups}

Let the pair $(X,G)$ denote a faithful right action of a group $G$ on a
set~$X$. We also say that $(X,G)$ is a \emph{permutation group}
because this action induces an injective homomorphism
from $G$ into the symmetric group of permutations on $X$ acting on the right.
The cardinal of $X$ is the \emph{degree} of the permutation group $(X,G)$.
Two permutation groups $(X,G)$ and $(Y,H)$ are said to be \emph{equivalent}
if there is a pair $(f,\varphi)$ formed by a bijection
$f\colon X\to Y$ and an isomorphism $\varphi\colon G\to H$
such that $f(x\cdot g)=f(x)\cdot\varphi(g)$, for every $x\in X$ and $g\in G$.
We then say that $(f,\varphi)\colon (X,G)\to (Y,H)$ is an equivalence.

In this section, we show that, loosely speaking,
if the conditions of Theorem~\ref{t:group-code-isomor-maximal-subgroups}
are satisfied, then, in a natural manner, the maximal subgroups
of $J_F(X)$ and $J(Z)$ are equivalent
permutation groups.

The following lemma will be useful to describe the bijection
on the sets where the groups act. Recall that $i_X$
is the initial state and unique final state of
the minimal automaton $\mathcal M_{X^\ast}$
of $X^\ast$, whenever $X$ is a prefix code.

\begin{Lemma}\label{l:action-of-the-group}
  Let $F$ be a recurrent subset of $A^\ast$.
  Suppose that $Z$
  is a rational bifix code of finite degree,
  and that the intersection $X=Z\cap F$ is
  also rational.
  Consider an $\mathcal L$-class $K$ of $J(F)$.
  For every $u,v\in K$, we have
  \begin{equation}\label{eq:action-of-the-group-0}
    i_X\cdot u=i_X\cdot v\Leftrightarrow
    i_Z\cdot u=i_Z\cdot v.
  \end{equation}
\end{Lemma}

\begin{proof}
  Because $J(F)$ is regular,
  we may take $u'\in J(F)$ such that $u=uu'u$.
  And since $u\mathrel{\mathcal L}v$,
  there is $t\in J(F)$ such that $v=tu$.

  Suppose that $i_Z\cdot u=i_Z\cdot v$. Then we have
  \begin{equation}
    \label{eq:action-of-the-group-1}
    i_Z\cdot uu'=i_Z\cdot vu'=i_Z\cdot tuu'.
  \end{equation}
  Note that $uu'$ is idempotent.
  Therefore, we
  have $uu'\in\overline{Z^*}$ by 
  Proposition~\ref{p:idempotents-are-in-the-closure-of-Xast}.
  Hence, \eqref{eq:action-of-the-group-1} simplifies
  to $i_Z=i_Z\cdot tuu'$,
  whence $tuu'\in\overline{Z^*}$. Because $uu'\in\overline{Z^*}$, applying Proposition~\ref{p:propositionUnitary} we get
  $t\in\overline {Z^*}$.
   And since $t\in\overline{F}$ and $X$ is rational,
   we conclude that $t\in\overline {X^*}$.
  Therefore,
  we have $i_X=i_X\cdot t$, whence
  $i_X\cdot u=i_X\cdot tu=i_X\cdot v$.
  
  Conversely, suppose that $i_X\cdot u=i_X\cdot v$. Then we have
  \begin{equation}
    \label{eq:action-of-the-group-2}
    i_X\cdot uu'=i_X\cdot vu'=i_X\cdot tuu'.
  \end{equation}
  Since $uu'$ is an idempotent factor of $u\in J(F)$ and $\overline{F}$
  is factorial, we know that
  $uu'\in\overline{X^*}$ by
  Proposition~\ref{p:idempotents-are-in-the-closure-of-Xast}.
  Hence, \eqref{eq:action-of-the-group-2} simplifies
  to $i_X=i_X\cdot tuu'$.
  Therefore, we have $uu',tuu'\in\overline{X^*}$,
  thus $t\in\overline{X^*}$ by Proposition~\ref{p:propositionUnitary}.
  This implies $i_Z=i_Z\cdot t$, whence 
  $i_Z\cdot u=i_Z\cdot tu=i_Z\cdot v$.
\end{proof}

Consider the set $Q$ of states
of the minimal automaton $\mathcal M_L$
of $L$, where~$L$ is a rational language. For each element $s$ of the syntactic monoid
$\Syn(L)$, let
\begin{equation*}
  Q\cdot s=\{q\cdot s\mid q\in Q\}.
\end{equation*}
In a transformation monoid, $\mathcal L$-equivalent elements have the same image. In particular, in the cases of interest for us,
when $K$ is an $\mathcal L$-class or
a subgroup of $\Syn(L)$,
the set
\begin{equation*}
  Q\cdot K=\{q\cdot g\mid q\in Q\text{ and }g\in K\},
\end{equation*}
is such that $Q\cdot K=Q\cdot g$ for all $g\in K$.
Moreover, if $K$ is a subgroup, then $K$ acts faithfully as a permutation group on $Q\cdot K$.
As it is well known, if $K$ and $K'$ are maximal subgroups of $\Syn(L)$ contained in the same $\mathcal J$-class,
then $(Q\cdot K,K)$ and $(Q\cdot K',K')$
are equivalent permutation groups.

Suppose that $X$ is a rational code. We denote by $Q_X$ the set of
states of the minimal automaton of $\mathcal M_{X^*}$.
When $u\in\F A$, then $Q_X\cdot u$ denotes $Q_X\cdot \hat\eta_{X^*}(u)$,
and if $K$ is an $\mathcal L$-class or a subgroup of $\F A$,
then $Q_X\cdot K$ denotes $Q_X\cdot \hat\eta_{X^*}(K)$,
so that in particular we
have $Q_X\cdot K=Q_X\cdot u$ whenever $u\in K$.
Finally, we denote by $i_X\cdot K$ the set
$\{i_X\cdot u\mid u\in K\}$.

\begin{Lemma}\label{l:iH}
  Let $F$ be a recurrent subset of $A^*$, and let
  $X$ be a rational bifix code contained in $F$ and with finite $F$-degree. 
  If $K$ is an $\mathcal L$-class of $J(F)$, then the equality
    $Q_X\cdot K=i_X\cdot K$
  holds.
\end{Lemma}

Our proof of Lemma~\ref{l:iH} depends on the following
property. In its statement, by ``a prefix of $X$'' we mean
``a prefix of some element of $X$''.

\begin{Lemma}[{\cite[Lemma 7.1.4]{Berstel&Felice&Perrin&Reutenauer&Rindone:2012}}]\label{l:lemma-714-from-jalgebra-paper}
  Let $F$ be a recurrent subset of $A^*$, and let
  $X$ be a bifix code contained in $F$ and with finite $F$-degree $d$.
  If $w$ is a word of $F$ such that $\delta_X(w)=d$,
  then, for every $q\in Q_X\cdot w$ there is a unique proper
  prefix $v$ of~$X$ which is a suffix of $w$, and such that
  $q=i_X\cdot v$.
\end{Lemma}

\begin{proof}[Proof of Lemma~\ref{l:iH}]
  Take $q\in Q_X$ and $u\in K$.
  Consider also an element $u'$ of $J(F)$ such that
  $u=uu'u$.  
  By Lemma~\ref{l:degree-in-JF},
  we have $\delta_X(u)=d_F(X)$.
Consider a sequence $(u_n)_n$ of elements of
$F$ converging in $\F A$ to $u$.
Taking subsequences,
we may as well suppose
that $\hat\eta_{X^\ast}(u_n)=\hat\eta_{X^\ast}(u)$
and $\delta_X(u_n)=\delta_X(u)=d_F(X)$
for all $n$,
respectively  because
of the continuity of $\hat\eta_{X^\ast}$
and of $\delta_X$ (cf.~Proposition~\ref{p:the-parse-is-continuous}).
In particular, we have $q\cdot u_n=q\cdot u$ for all $n$.
Moreover, by Lemma~\ref{l:lemma-714-from-jalgebra-paper}, for each $n$ there is a unique proper prefix $v_n$
of $X$ such that $v_n$ is a suffix of $u_n$
and $q\cdot u_n=i_X\cdot v_n$.
Let $v$ be an accumulation point of the sequence $(v_n)_n$.
Then,
$v$ is a suffix of $u$
and $q\cdot u=i_X\cdot v$.
Multiplying by $u'u$,
we get $q\cdot u=i_X\cdot vu'u$.
Since $v$ is a suffix of $u$,
we know that $vu'u$ is a suffix of $uu'u=u$,
and so $vu'u$ and $u$ are $\mathcal L$-equivalent.
This shows that $q\cdot u\in i_X\cdot K$, thus establishing
the equality $Q_X\cdot K=i_X\cdot K$.
\end{proof}

We are ready for the main result of this section. 
Note that, in view
of Theorem~\ref{t:group-code-isomor-maximal-subgroups},
the group $H_X$ in the next statement
is indeed a maximal subgroup of $J_F(X)$.

\begin{Thm}\label{t:equivalence-of-permutation-groups}
  Consider a recurrent subset $F$ of $A^*$.
  Suppose that $Z$
  is an $F$-charged rational bifix code of finite degree $d$. Let $X=Z\cap F$ and suppose
  that $X$ is rational.
  Let $H$ be a maximal subgroup of $J(F)$.
  Consider the maximal
  subgroup $H_Z=\hat\eta_{Z^\ast}(H)$ of $J(Z)$
  and the maximal subgroup $H_X=\hat\eta_{X^\ast}(H)$
  of $J_F(X)$.
  Denoting by $\mathcal L(H)$ the $\mathcal L$-class
  of $J(F)$ containing~$H$,
  take the correspondence $f$ from $Q_X\cdot H_X$ to $Q_Z\cdot H_Z$
  defined by
  \begin{equation*}
    f(i_X\cdot u)=i_Z\cdot u,\quad\forall u\in \mathcal L(H),
  \end{equation*}
  together with the unique isomorphism $\alpha\colon H_X\to H_Z$ such that
  the diagram
  \begin{equation*}
    \xymatrix{
      H\ar[r]^{\hat\eta_{Z^\ast}}\ar[d]_{\hat\eta_{X^\ast}}&H_Z\\
      H_X\ar[ur]_\alpha&
    }
  \end{equation*}
  commutes. Then $(f,\alpha)\colon (Q_X\cdot H_X,H_X)\to (Q_Z\cdot H_Z,H_Z)$
  is an equivalence between permutation groups
  of degree $d$.
\end{Thm}

\begin{proof}
  Since
  $Q_X\cdot H=Q_X\cdot \mathcal L(H)$
  and
  $Q_Z\cdot H=Q_Z\cdot \mathcal L(H)$,
  the correspondence $f$ is well defined and injective
  by the combination of
  Lemmas~\ref{l:action-of-the-group} and~\ref{l:iH}.
  According to Proposition~\ref{p:d-parses-implies-you-are-in-JFX},
  we have $|Q_X\cdot H|=d_F(X)$ and $|Q_Z\cdot H|=d(Z)$.
  Applying Theorem~\ref{t:group-code-isomor-maximal-subgroups},
  we get that $|Q_X\cdot H|=|Q_Z\cdot H|=d$, establishing
  in particular that $f$ is a bijection.
  The existence and uniqueness of the isomorphism $\alpha$
  also follows from Theorem~\ref{t:group-code-isomor-maximal-subgroups}
  and its proof, based on Theorem~\ref{t:injective-homomorphism-maximal-subgroups}. Finally,
  let $q\in Q_X\cdot H$ and $g\in H_X$.
  Then, there are $u\in \mathcal L(H)$ and $v\in H$ such that
  $q=i_X\cdot u$ and $g=\hat\eta_{X^*}(v)$, the former because
  of Lemma~\ref{l:iH}.
  Noting that we also have $uv\in\mathcal L(H)$,
  we then have
  \begin{equation*}
    f(q)\cdot \alpha(g)
    =(i_Z\cdot u)\cdot\hat\eta_{Z^*}(v)
    =i_Z\cdot uv
    = f(i_X\cdot uv)
    =f(q\cdot g),
  \end{equation*}  
 thus concluding the proof.
\end{proof}

As an example of application of Theorem~\ref{t:equivalence-of-permutation-groups},
we next deduce a version of Corollary~\ref{t:equal-degree-connected-set-group-code} for permutation groups. Recall that a permutation group
$(Y,G)$ is \emph{transitive}
when, for all $x,y\in Y$, there is $g\in G$ such that $x\cdot g=y$.
In the following statement, we consider the natural action
of a maximal subgroup of $\Syn(X^*)$ on the minimal automaton
of $X^*$.

\begin{Cor}\label{c:transitivity-of-GFX}
  Consider a group code $Z$ of $A^*$ and let $F$
  be a uniformly recurrent connected set with alphabet $A$.
  Take the intersection~\mbox{$X=Z\cap F$}.
  The permutation groups
  $G(Z)$ and $G_F(X)$ are isomorphic transitive
  permutation groups of degree~$d(Z)$.
\end{Cor}

\begin{proof}
  The group $G(Z)$ acts transitively on the
  minimal automaton of $Z^*$
  as a permutation group of degree $d(Z)$, because the minimal automaton
  of $Z^*$ is trim and $G(Z)$ is the syntactic monoid of $Z^*$.
  Since transitivity is preserved by equivalence of permutation groups, the result follows immediately
  from Theorem~\ref{t:equivalence-of-permutation-groups},
  together with Proposition~\ref{p:group-codes-are-connected-charged}
  and Theorem~\ref{t:group-code-isomor-maximal-subgroups}.
\end{proof}

In general, for an arbitrary $F$-maximal bifix code $X$ with
$F$ uniformly recurrent, the natural action of $G_F(X)$ may not be transitive
(cf.~\cite[Example 7.2.1]{Berstel&Felice&Perrin&Reutenauer&Rindone:2012}).

We end this section with
a description of $F$-groups
as permutation groups defined by suitable sets of first return words,
under general conditions that
in particular apply to the setting
of Theorem~\ref{t:equivalence-of-permutation-groups}.

\begin{Prop}\label{p:description-of-permutation-groups}
  Consider a recurrent subset $F$ of $A^*$.
  Let $X$ be a weakly $F$-charged
  code of $A^*$.
  Let~$K$ be a maximal subgroup of~$J_F(X)$.
  Suppose that $u\in F$ is such that
  $\eta_{X^*}(u)$ is $\mathcal L$-equivalent
  to the elements of~$K$.
  For each $v\in R_F(u)$, let $\pi_v$
  be the permutation of $Q_X\cdot K$ resulting
  from the restriction to $Q_X\cdot K$ of
  the action of $v$ in the minimal automaton $\mathcal M_{X^*}$ of $X^*$.
  Then the group of permutations
  of $Q_X\cdot K$ generated by $\{\pi_v\mid v\in R_F(u)\}$
  is the group of permutations resulting
  from the restriction to $Q_X\cdot K$ of
  the action in $\mathcal M_{X^*}$ of the elements of~$K$.
\end{Prop}

\begin{proof}
  We claim that there is some maximal subgroup $H$ of $J(F)$
  such that $u$ is a suffix of the elements of $H$.
  Take an idempotent $e$ of $J(F)$.
  Since $u\in F$, there are $x,y\in\F A$ with $e=xuy$.
  From $e=(exuy)^3$ we get that both
  $yexu$ and $(yexu)^2$ are $\mathcal J$-equivalent to $e$,
  and so $yexu$ belongs to a maximal subgroup $H$ of~$J(F)$,
  establishing the claim.

  By the hypothesis
  that $X$ is weakly $F$-charged,
  the image $H_X=\hat\eta_{X^*}(H)$
  is a maximal subgroup of $J_F(X)$.
  Observe also that $K$ and $H_X$ are maximal subgroups contained in the
  same $\mathcal L$-class, namely the one that contains $\eta_{X^*}(u)$,
  and that therefore
  the permutations of $Q_X\cdot H_X$ induced by elements of $H_X$
  are precisely 
  the permutations of $Q_X\cdot K=Q_X\cdot H_X$ induced by elements of $K$.

  Let $v\in R_F(u)$.
  As $u$ is a suffix of $v$ and of the elements of $H$,
  and since $\eta_{X^*}(u)\in J_F(X)$, $uv\in F$ and $H_X\subseteq J(F)$,
    we know that $\eta_{X^*}(uv)$ and $\eta_{X^*}(u)$
    both belong to the $\mathcal L$-class of $\Syn(X^*)$ containing $H_X$.
    In particular,
    we have $Q_X\cdot uv=Q_X\cdot u=Q_X\cdot H_X$.
    Therefore, $(Q_X\cdot H_X)\cdot v=Q_X\cdot H_X$ holds.
    This means that $v$ acts as a permutation
    on the finite set $Q_X\cdot H_X$.
    Consider the idempotent $f$ in $H$.
    Then
    we have $\hat\eta_{X^*}(f)\mathrel{\mathcal L}\eta_{X^*}(u)$,
    whence $\hat\eta_{X^*}(fv)\mathrel{\mathcal L}\eta_{X^*}(uv)$.
    But we saw that we also have
    $\eta_{X^*}(u)\mathrel{\mathcal L}\eta_{X^*}(uv)$,
    and so $\hat\eta_{X^*}(f)\mathrel{\mathcal L}\hat\eta_{X^*}(fv)$.
    Since $\Syn(X^*)$ is stable, the latter implies
    that $\hat\eta_{X^*}(fv)$ belongs to the maximal subgroup containing
    $\hat\eta_{X^*}(f)$, that is, to $H_X$.
    Clearly, the action of $v$ in $Q_X\cdot H_X$ is the same
    as the action of $fv$ in $Q_X\cdot H_X$,
    and so the permutation $\pi_v$ of $Q_X\cdot H_X$
    induced by $v$ is a permutation induced by an element of $H_X$, namely
    $\hat\eta_{X^*}(fv)$.
    
    Conversely, let $h\in H_X$.
  Then $h=\hat\eta_{X^*}(w)$ for some $w\in H$.
  Because we also have $w^2\in H$ and the elements of $H$ have $u$ as a suffix,
  we know that $uw\in J(F)$. Hence $uw\in\overline{F}$,
  and  we may consider a sequence of words of $F$
  converging to~$uw$.
  Applying Lemma~\ref{l:refine-factorization}
  together with the continuity of
  $\hat\eta_{X^*}$, we conclude that
  there is some finite word $w'$ such that
  $h=\hat\eta_{X^*}(w)=\eta_{X^*}(w')$, $uw'\in F$ and $w'\in A^\ast u$.
  In particular, $w'\in R_F(u)^+$,
  and $h$ belongs
  to the subsemigroup of $\Syn(X^*)$ generated by $\eta_{X^*}(R_F(u))$.
   Therefore, the permutation on $Q_X\cdot H_X$ induced by
   $h$ belongs to the group of permutations
   of $Q_X\cdot H_X$ generated by
   $\{\pi_v\mid v\in R_F(u)\}$.
\end{proof}

\begin{Example}\label{eg:exampleDegree2-continuation}
  Let us go back to Example~\ref{eg:exampleDegree2}.
  In this example, we have $R_F(a)=\{a,ba\}$. The
  image of $a$ and $ba$ on the minimal automaton of~$X^*$,
  represented in Figure~\ref{figureAutomaton}, is the set $\{1,2\}$,
  and the set of corresponding
  permutations, respectively the transposition $(1,2)$ and
  the identity, generates the group or order $2$,
  as predicted by Proposition~\ref{p:description-of-permutation-groups}.
\end{Example}

\begin{Rmk}\label{eg:not-H-equivalent}
  In Proposition~\ref{p:description-of-permutation-groups},
  while $\{\pi_v\mid v\in R_F(u)\}$
  generates a group isomorphic to a maximal subgroup of $J_F(X)$,
  one may have distinct elements of $R_F(u)$
  being mapped into distinct $\mathcal H$-classes of $J_F(X)$:
  see Example~\ref{eg:exampleDegree2-continuation},
  where $\eta_{X^*}(a)$ and $\eta_{X^*}(ba)$
  are not $\mathcal H$-equivalent,
  as one can see in the egg-box diagram of Figure~\ref{figureAutomaton}.
\end{Rmk}

We end with an example
concerning an $F$-charged group code
such that $F$ is a non-connected uniformly
recurrent set.

\begin{Example}\label{eg:exampleThueMorse8}
  Let $A=\{a,b\}$.
  Consider the \emph{Prouhet-Thue-Morse substitution} $\tau\colon A^*\to A^*$,
  defined by $\tau(a)=ab$ and $\tau(b)=ba$,
  and let $F=F(\tau)$.
  The set $F$ is not connected: the extension graph of $aba$
  with respect to $F$ has two connected components,
  since $A^*abaA^*\cap F=\{aabab,babaa\}$.
  
  Let $h$ be the homomorphism
  $h:a\mapsto (123),b\mapsto (345)$ from $A^*$ onto the alternating
  group~$A_5$. It is shown in~\cite{Almeida&ACosta:2013}
  that its unique continuous extension
  $\hat h\colon \F A\to A_5$
  maps each maximal subgroup of $J(F)$
  onto $A_5$ (cf.~\cite[proof of Remark 7.8]{Almeida&ACosta:2013}). 

  Let $Z$ be the group code generating the submonoid of $A^*$ stabilizing $1$ via the action induced by $h$.
  This action describes the group automaton
  of~$Z^*$, with vertex set $\{1,2,3,4,5\}$
  and initial set $1$, and so~$h$ is precisely the syntactic homomorphism
  $\eta_{Z^*}$. Therefore, $Z$ is $F$-charged.
  It follows from Theorem~\ref{t:group-code-isomor-maximal-subgroups}
  that the
  $F$-maximal bifix code $X=Z\cap F$
  is weakly $F$-charged.
  The code $X$ is represented in
Figure~\ref{figureBifixDegree5Morse}.
We represent in Figure~\ref{figureBifixDegree5Morse}
only the nodes corresponding
to \emph{right special words}, that is, vertices with two sons.
\begin{figure}[hbt]
  \centering
{\tiny
  \begin{tikzpicture}[x=1.05mm,y=1.2mm,shorten >=1pt,on grid,semithick]
    \tikzstyle{every node}=[draw,rectangle,rounded corners,minimum size=0.8mm]
    \draw (0,10) node (1) {$1$};
    \draw (20,30) node (2) {$2$};
    \draw (20,0) node (b) {$1$};
    \draw (30,40) node (3) {$3$};
    \draw (30,20) node (4) {$4$};
    \draw (45,50) node (x1) {$1$};
    \draw (45,40) node (5) {$5$};
    \draw (45,25) node (6) {$6$};
    \draw (45,15) node (7) {$7$};
    \draw (60,50) node (x2) {$1$};
    \draw (60,40) node (x3) {$1$};
    \draw (60,33) node (x4) {$1$};
    \draw (60,25) node (8) {$8$};
    \draw (60,15) node (x9) {$1$};
    \draw (60,5) node (9) {$9$};
    \draw (80,30) node (10) {$10$};
    \draw (80,20) node (x8) {$1$};
    \draw (80,10) node (11) {$11$};
    \draw (80,0) node (x13) {$1$};
    \draw (100,40) node (x5) {$1$};
    \draw (100,30) node (12) {$12$};
    \draw (100,20) node (x10) {$1$};
    \draw (100,10) node (12b) {$12$};
    \draw (120,35) node (x6) {$1$};
    \draw (120,25) node (x7) {$1$};
    \draw (120,15) node (x11) {$1$};
    \draw (120,5) node (x12) {$1$};

    \tikzstyle{every node}=[inner sep=2pt]
    \path[-] (1) edge node [above left] {$a$} (2)
             (1) edge node [above] {$b$} (b)
             (2) edge node [above left] {$ab$} (3)
             (2) edge node [above right] {$b$} (4)
             (3) edge node [above left] {$abba$} (x1)
             (3) edge node [below] {$ba$} (5)
             (4) edge node [above left] {$a$} (6)
             (4) edge node [below] {$ba$} (7)
             (5) edge node [above left] {$aba$} (x2)
             (5) edge node [below] {$ba$} (x3)
             (6) edge node [above left] {$a$} (x4)
             (6) edge node [below] {$b^2a$} (8)
             (7) edge node [above] {$a$} (x9)
             (7) edge node [below left] {$\tau^2(b)$} (9)
             (8) edge node [above, inner sep=5pt] {$\tau^3(a)$} (10)
             (8) edge node [below] {$ba$} (x8)
             (9) edge node [above, inner sep=5pt] {$\tau^2(a)$} (11)
             (9) edge node [below] {$ba$} (x13)
             (10) edge node [above left] {$\tau^2(a)$} (x5)
             (10) edge node [below] {$\tau^3(b)$} (12)
             (11) edge node [above left] {$\tau^2(a)$} (x10)
             (11) edge node [below] {$\tau^2(bba)$} (12b)
             (12) edge node [above, inner sep=5pt] {$\tau^2(a)$} (x6)
             (12) edge node [below] {$\tau^2(b)a$} (x7)
             (12b) edge node [above, inner sep=5pt] {$\tau^2(a)$} (x11)
             (12b) edge node [below] {$\tau^2(b)a$} (x12);
           \end{tikzpicture}
         }
\caption{The bifix code $X$ from Example~\ref{eg:exampleThueMorse8}.}\label{figureBifixDegree5Morse}
\end{figure}
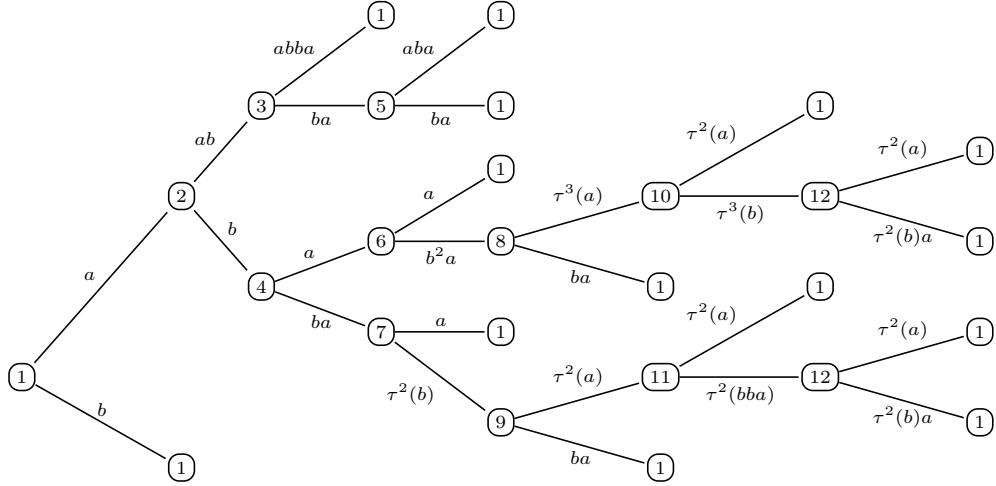
The image of $\tau^4(b)$ in $\mathcal M_{X^*}$ is $\{1,3,4,9,10\}$.
Since $\tau^4(b)$ has rank $5$ in $\mathcal M_{X^*}$
and $d_F(X)=d(Z)=5$ by Theorem~\ref{t:group-code-isomor-maximal-subgroups},
we know that the image of $\tau^4(b)$ in $\mathcal M_{X^*}$
belongs to the $F$-minimum $\mathcal J$-class $J_F(X)$.
The action
of $\tau^4(b)$ on its image is shown in Figure~\ref{figureMinimalImages}.
The return words to $\tau^4(b)$ are $\tau^4(b),\tau^3(a)$ and
$\tau^5(ab)$. The permutations on the image of $\tau^4(b)$ are the three
cycles of length~$5$ indicated in Figure~\ref{figureMinimalImages}.
They generate the group $G_F(X)=A_5=G(Z)$, in agreement with
Theorem~\ref{t:group-code-isomor-maximal-subgroups}
and Proposition~\ref{p:description-of-permutation-groups}.

\begin{figure}[hbt]
  \centering
  \begin{tikzpicture}[x=10mm,y=10mm,shorten >=1pt,on grid,semithick,
    node distance=5cm]
    \footnotesize
    \tikzstyle{every node}=[draw,rectangle,rounded corners,minimum size=0.8mm]
    \node (1) {$\{1,3,4,9,10\}$};
    \node (2) [right=of 1] {$\{1,2,7,8,12\}$};
    \tikzstyle{every node}=[]
    \path[->]
    (1) edge [loop left] node [left] {$\tau^4(b)\mid(1,9,10,3,4)$}
    (1) edge node [above] {$\tau^4(a)$} (2)
    (2) edge [bend right=30] node [above] {$\tau^3(a)\mid (1,10,9,3,4)$} (1)
    (2) edge [bend left=30]  node [below] {$\tau^4(b)\mid (1,10,9,4,3)$} (1);
  \end{tikzpicture}
\caption{The action on the image of $\tau^4(b)$.}\label{figureMinimalImages}
\end{figure}
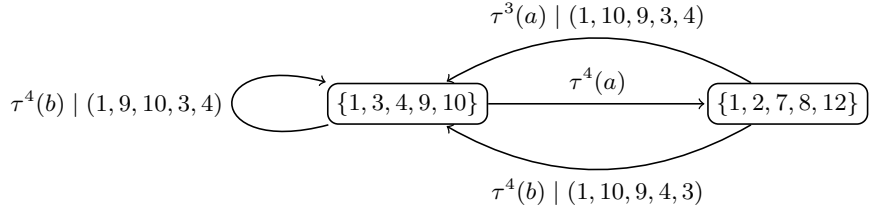

\end{Example}

\bibliographystyle{amsalpha}

\newcommand{\etalchar}[1]{$^{#1}$}
\providecommand{\bysame}{\leavevmode\hbox to3em{\hrulefill}\thinspace}
\providecommand{\MR}{\relax\ifhmode\unskip\space\fi MR }
% \MRhref is called by the amsart/book/proc definition of \MR.
\providecommand{\MRhref}[2]{%
  \href{http://www.ams.org/mathscinet-getitem?mr=#1}{#2}
}
\providecommand{\href}[2]{#2}

\end{document}